          \documentclass{cmslatex}
\usepackage[paperwidth=7in, paperheight=10in, margin=.875in]{geometry}

\renewcommand{\theequation}{\arabic{section}.\arabic{equation}}

\newcommand{\f}[2]{\frac{#1}{#2}}
\newcommand{\dpr}[2]{\langle #1,#2 \rangle}

\newcommand{\al}{\alpha}
\newcommand{\be}{\beta}

\newcommand{\ga}{\gamma}

\newcommand{\de}{\delta}

\newcommand{\De}{\Delta}

\newcommand{\ka}{\kappa}

\newcommand{\La}{\Lambda}

\newcommand{\vp}{\varphi}

\newcommand{\rd}{{\mathbf R}^d}

\newcommand{\rone}{\mathbf R^1}
\newcommand{\rtwo}{\mathbf R^2}

\newcommand{\cm}{\mathcal M}

\newcommand{\p}{\partial}

\newcommand{\beq}{\begin{equation}}
\newcommand{\eeq}{\end{equation}}
\newcommand{\beqna}{\begin{eqnarray*}}
\newcommand{\eeqna}{\end{eqnarray*}}
\newcommand{\beqn}{\begin{equation*}}
\newcommand{\eeqn}{\end{equation*}}
\newcommand{\bp}{\begin{proof}}
\newcommand{\ep}{\end{proof}}
\newcommand{\bprop}{\begin{proposition}}
\newcommand{\eprop}{\end{proposition}}
\newcommand{\bt}{\begin{theorem}}
\newcommand{\et}{\end{theorem}}
\newcommand{\bex}{\begin{Example}}
\newcommand{\eex}{\end{Example}}
\newcommand{\bc}{\begin{corollary}}
\newcommand{\ec}{\end{corollary}}
\newcommand{\bcl}{\begin{claim}}
\newcommand{\ecl}{\end{claim}}
\newcommand{\bl}{\begin{lemma}}
\newcommand{\el}{\end{lemma}}

\begin{document}
          \title{On the global regularity  of  the 2D critical Boussinesq  system  with  $\al>2/3$\thanks{HaddadiFard has been partially supported  by a  graduate research fellowship through a grant NSF-DMS \#  1313107. Stefanov's research has been supported in part by NSF-DMS  \#  1313107 and  \# 1614734.}}


          \author{Fazel Hadadifard\thanks{ Department of Mathematics,
University of Kansas,
1460 Jayhawk Blvd,
Lawrence, Kansas 66045-7594, f.hadadi@ku.edu}
         \and{Atanas Stefanov\thanks{Department of Mathematics,
University of Kansas,
1460 Jayhawk Blvd,
Lawrence, Kansas 66045-7594, stefanov@ku.edu}}}





         \pagestyle{myheadings} \markboth{Regularity of the critical 2D Boussinesq Equations}{Fazel Hadadifard, Atanas Stefanov} \maketitle

          \begin{abstract}
           This paper examines  the question for global regularity for the Boussinesq equation with critical fractional dissipation $(\al, \be):\al+\be=1$. 
 The main result states  that the system admits global regular solutions for all (reasonably) smooth and decaying data, as long as $\al>2/3$. This improves upon 
 some recent works \cite{JMWZ} and \cite{SW}. 
 
 The main new idea  is the introduction of a new, second generation Hmidi-Keraani-Rousset type, change of variables, which further improves the linear derivative in temperature term in the vorticity equation. This approach is then complemented by  new set of  commutator estimates (in both negative and positive index Sobolev spaces!), which may be of independent interest.          
                     \end{abstract}
\begin{keywords}
Boussinesq equations, fractional dissipation, global regularity
\end{keywords}

 \begin{AMS}
 35Q35, 35B65, 76B03
\end{AMS}
       \section{Introduction}
The two-dimensional (2D) Boussinesq equations with fractional dissipation is 
\begin{equation}
\label{BQE}
\left\{
\begin{array}{l}
\partial_t u + u\cdot \nabla u + \nu\, \Lambda^\alpha u =- \nabla p + \theta \mathbf{e}_2, \qquad x\in \rtwo, \,\, t>0, \\
\nabla \cdot u=0, \qquad x \in \rtwo, \,\, t>0, \\
\partial_t \theta + u\cdot \nabla \theta
+ \kappa\, \Lambda^{\beta}\theta =0,\qquad x\in \rtwo, \,\, t>0,\\
u(x,0) =u_0(x),\,\, \theta(x,0) =\theta_0(x), \qquad x\in \rtwo,
\end{array}
\right.
\end{equation}
where $u=u(x,t)=(u_1(x,t),u_2(x,t))$ denotes the  velocity vector field, $p=p(x,t)$ is the scalar  pressure,
the scalar function $\theta=\theta(x,t)$ is the temperature, $\mathbf{e}_2$ the unit vector in the vertical
direction, and $\nu\geq 0$, $\kappa\geq 0$, $0\leq \alpha \le 2$ and $0\leq \beta\le 2$
are real parameters. Here $\Lambda= \sqrt{-\Delta}$ is  the Zygmund
operator  defined through the Fourier transform,
$$
\widehat{\Lambda^\alpha f}(\xi) = |\xi|^\alpha \,\widehat{f}(\xi),
$$
where the Fourier transform and its inverse are given by
$$
\widehat{f}(\xi) = \int_{\\rtwo} e^{-i x\cdot \xi} \, f(x) \,dx, \ \ f(x)=(2\pi)^{-2}  
\int_{\\rtwo} e^{i x\cdot \xi} \, \hat{f}(\xi) \,d\xi. 
$$
This model is of importance in a number of studies on atmospheric turbulence, \cite{MB,P}.  The standard model (where both dissipations are taken to be the classical Laplacian, $\al=\be=2$) is a primary model for atmospheric fronts and oceanic circulation as well as in the study of Raleigh-Bernard convection, 
\cite{Con_D, Gill, Maj, P, Doering1, Doering2}. The fractional diffusion operators considered herein appear naturally in the study in hydrodynamics, \cite{DI} as well as anomalous diffusion in semiconductor growth, \cite{PS}. There are also other models in which the Boussinesq equations with fractional Laplacian naturally arise, namely in models where the kinematic and thermal diffusion is attenuated by the thinning of atmosphere,  \cite{Gill}. 

Mathematically, the problem for global regularity of \ref{BQE} is an interesting and a subtle one. Intuitively, the lower the values of $\al, \be$, the harder it is to prove that solutions emanating from sufficiently smooth and localized data persist globally. In particular, the problem with  no dissipation (i.e. $\nu=\ka=0$) remains open. This is very similar to the Euler equation in two and three spatial dimensions  and in fact numerous studies explore the possibility of finite time blow up, \cite{SaWu}. 

Next, we  take on the difficult task of reviewing the recent results regarding well-posedness issues for \ref{BQE}. Indeed, there has been tremendous interest in this problem in the last fifteen years. In the classical case, when the diffusion is given by the regular Laplacian (i.e.  $\al=\be=2$), the global regularity follows just as it does for the 2D Navier-Stokes model, \cite{DoeringG, MB}. In the works \cite{Ch, HL}, global regularity was proved in the presence of one full Laplacian, that is in the cases  $\al=2, \be=0$ or $\al=0, \be=2$. 
In more recent years, the full two parameter range of $\al, \be$ was explored in detail. Based on the currently available results, it is natural to draw the  conclusion that one expects global regularity in the cases $\al+\be\geq 1$, while the case $\al+\be<1$ generally remains open\footnote{in some numerical simulations, there was a reason to believe that finite time blow up might occur, but this is at present still a conjecture}.  We thus adopt  the   notion of criticality - 
namely, we say that a pair $(\al, \be)$ is subcritical if $\al+\be>1$, critical if $\al+\be=1$ and supercritical if $\al+\be<1$. 

As it was alluded above, in the supercritical regime the behavior of the solutions  remains a mystery. Apart from some numerical simulations,  the only rigorous results that we are aware of  is the eventual regularity of the solutions, \cite{Wu_Xu_Ye}, for appropriate supercritical regime of the diffusivity parameters.   To be sure, such statement does not, {\it per se} exclude a finite time blow up of some solutions. It remains to discuss the critical and subcritical cases. This is probably a good place to observe that if global regularity holds for critical pair $(\al_0, \be_0): \al_0+\be_0=1$, then it must hold for all subcritical pairs in the form\footnote{We believe that this statement, while not a rigorous result, can be made an exact theorem on a case by case basis, by just reworking a proof for $(\al_0, \be_0)$ to cover the higher dissipation cases (in either the $u$ or the $\theta$ variables)}  $(\al, \be_0), \al>\al_0$ and $(\al_0, \be): \be>\be_0$. Thus, clearly global regularity results on the critical line are superior, in the sense described above, to subcritical ones. That being said, the subcritical theory is far from obvious or well-understood. Many results have been put forward in the last ten or so years. The following (very incomplete and yet very long) list accounts for some recent accomplishments - \cite{MX, CV, YangJW, YXX, Zhao, YX_15, YX_16, YEX, Cha1, Dan, LLT, LaiPan, Xu, Ye_16, Ye_17}. 

Next, we give a full account of the global regularity results for diffusivity parameters on the critical line $\al+\be=1$.  First, in series of works, Hmidi-Keraani-Rousset, \cite{HKR1, HKR2} established global regularity in the two critical and endpoint cases $(\al, \be)=(1,0), (0,1)$. In their work, they employed clever change of variables, thus introducing a new hybrid quantity, depending on both vorticity and temperature\footnote{which is better suited (and looses less derivatives than either vorticity and temperature separately}. 
In a subsequent paper, by developing more sophisticated  function spaces, Jiu-Miao-Wu-Zhang, \cite{JMWZ}  were able to extend the global regularity results to the case $\al+\be=1$, 
$$
\al>\f{23-\sqrt{145}}{12}\approx 0.9132...
$$
Subsequently, the second name author, in collaboration with J. Wu, \cite{SW} significantly extended the results in \cite{JMWZ}, by covering the critical line $\al+\beta=1$, up to 
$$
\al>\f{\sqrt{1777}-23}{24}\approx 0.798103...
$$
Quite recently\footnote{after major part of this paper was completed}, we have learned that in \cite{WXXY}  Wu-Xu-Xue-Ye have managed to further lower the allowable $\alpha$ exponents to 
$$
\al>\f{10}{13}\sim 0.7692.
$$
These results were achieved thanks to more sophisticated commutator estimates, both in Sobolev and Besov spaces, by essentially working in the setup of Hmidi-Keraani-Rousset (HKR for short), \cite{HKR1}. It was our (informal) conclusion in \cite{SW} that tightening of the  commutator estimates in the HKR 
 variables has exhausted (or nearly exhausted) the possible improvements.  In other words, one needs to introduce better, more sophisticated change of variables, which in conjunction with sharp commutator estimates yields wider range of critical indices $(\al, \be)$, for which one has global regularity. 
 
 The purpose of this paper is to do just that. We aim at further improving upon the results in \cite{SW}. In particular, we still work in the regime\footnote{noting that the HKR framework takes a slightly different form in the case $\be>\f{1}{2}>\al$} $\al>\f{1}{2}>\be$, but in order to obtain better range, we perform a second generation HKR change of variables, which positions us for a better result. As we mention above, this is complemented by very precise commutator estimates, see Section \ref{sec:2.2}. 
 
 We note that  we do not, at this point, have anything new to say in the regime $\be>\f{1}{2}>\al$, for which the only available global regularity result is for $\al=0, \be=1$. We hope to be able to report on these cases in the near future. 
\subsection{Main result}
We are ready to state our main result. 
\begin{theorem}
\label{theo:1}
Consider the Boussinesq equation $\ref{BQE}$ with  
\[
 \frac{2}{3} < \alpha < 1, \alpha+ \beta= 1.
\]
Suppose also that 
$$
\|u_0\|_{H^{1+\rho}}<\infty, \ \ \|\theta_0\|_{H^{1 + \be+\rho}(R^2)}+\|\nabla\theta_0\|_{L^\infty}<\infty,
$$
where $0 < \rho << 1$. Then,  $\ref{BQE}$ has a unique global solution $(u, \theta)$ satisfying, for any $T > 0$ and moreover 
\[
(u, \theta) \in C ([0, T]; H^{1+\rho}(R^2)  \times H^{2+ \rho}(R^2)).
\]
\end{theorem}
{\bf Note:} We believe that the results in Theorem \ref{theo:1} are optimal, if one uses the second generation HKR transformations, as explained below in \ref{W1} below. In order to get improvements in the range $\al\in (1/2, 2/3)$, one needs to perform a third order HKR transformation and so on. 
\subsection{Some initial reductions}
It is well-known that for sufficiently smooth and decaying data, the problem has a local solution, say in some interval $[0,T]$. The global regularity problem then reduces to showing that $T=\infty$. One proceeds to establish that by a contradiction argument. That is, if one assumes that $T<\infty$, the contradiction will arise out of impossibility of blow up at time $T$. Thus, one seeks to prove {\it a priori} estimates on the solutions, which will prevent them from blowing up.  Let us mention for now, that the problem allows for some elementary {\it a priori} estimates 
\begin{eqnarray}{\label{fazel2}}
\left\{
  \begin{array}{ll}
\|\theta(t)\|_{L^p} \leq \|\theta_0\|_{L^p}, \textup{for}~~ p \in [1, \infty],\\
\|\theta(t)\|_{L^2}^2+ 2 \kappa \int_0^t \|\Lambda^{\frac{\beta}{2}} \theta(\tau)\|_{L^2}^2 d \tau = \|\theta_0\|_{L^2}^2,\\
\| u(t)\|_{L^2}^2+ 2 \nu \int_0^t \|\Lambda^{\frac{\alpha}{2}} u(\tau)\|_{L^2}^2 d\tau \leq ( \|u_0\|_{L^2}^2+ t \|\theta_0\|_{L^2}^2).
\end{array}
\right.
\end{eqnarray}
which are valid,  whenever $0<t<T$. These will be used  repeatedly in the argument, but as such they will be inadequate to conclude global regularity, they are just to weak for that.  From now on, due to the fact that the precise values of the physical constants $\ka, \nu>0$ are unimportant in the arguments, we set them to one, $\ka=\nu=1$.

\subsection{Change of variables: vorticity equation and beyond} 
It turns out that it is  easier to work with   the  vorticity equation.  A quick inspection shows that the vorticity $\omega = \nabla \times u$, a scalar quantity,  satisfies
\begin{eqnarray}{\label{fazel22}}
\left\{
  \begin{array}{ll}
\partial_t \omega+ u . \nabla \omega+  \Lambda^{\alpha} \omega= \partial_1 \theta,\\
u= \nabla^{\perp} \psi, ~~~~\Delta \psi= \omega ~~~~ \textup{or}~~~ u= \nabla^{\perp} \Delta^{-1} \omega
\end{array}
\right.
\end{eqnarray}
Therefore the problem reduces to the problem of considering the regularity and existence of the classical solution of the equations
\begin{eqnarray}{\label{fazel3}}
\left\{
  \begin{array}{ll}
    \omega_t+ \Lambda^{\alpha} \omega+ u \cdot\nabla \theta= \partial_1 \theta\\
     \theta_t+ \Lambda^{\beta} \theta + u \cdot \nabla \theta = 0
\end{array}
\right.
\end{eqnarray}
One notices of course, that the right-hand side of the vorticity equation has a full derivative acting on $\theta$, which is challenging to control. The strategy (first applied by Hmidi-Keraani-Rousset, \cite{HKR1}) is to consider a combined quantity of the vorticity and (a derivative of) the temperature $\theta$, which one would eventually be able to control via energy estimates.  
More precisely, note that since we can write 
$$
\Lambda^{\alpha} \omega-\partial_1 \theta= 
\Lambda^{\alpha}[ \omega-\Lambda^{-\al}\partial_1 \theta],
$$
 it is worth introducing the quantity $G= \omega-\Lambda^{-\alpha} \partial_1 \theta$. For it, we have the equation, 
$$
 G_t+ u \cdot \nabla G+ \Lambda^{\alpha} G= \Lambda^{\beta- \alpha} \partial_1 \theta+ [R_{\alpha}, u\cdot \nabla] \theta.
$$
This is the evolution equation used in \cite{HKR1} and subsequent papers, \cite{JMWZ}, \cite{SW}, \cite{WXXY}.  It turns out however that the presence of the factor $\Lambda^{\beta- \alpha} \partial_1 \theta$ is still too rough  in the range of $\al>\f{2}{3}$, thus preventing  us from getting the desired bounds. 
 In order  to remove it as is done above in the $G$ construction,  we introduce a new variable $f= G- \Lambda^{\beta- 2 \alpha} \partial_1 \theta$. This is the second generation HKR change of variables that we have alluded to above. 
 We have 
$$
G_t+ u \cdot \nabla G+ \Lambda^{\alpha} (G- \Lambda^{\beta- 2 \alpha} \partial_1 \theta) = [R_{\alpha}, u\cdot \nabla] \theta.
$$
Again by adding and subtracting some terms and using the equation for $\theta$, we get
\begin{eqnarray*}
(G- \Lambda^{\beta- 2 \alpha} \partial_1 \theta)_t  &+&  u . \nabla (G- \Lambda^{\beta- 2 \alpha} \partial_1 \theta)+ \Lambda^{\alpha} (G- \Lambda^{\beta- 2 \alpha} \partial_1 \theta)+  \Lambda^{\beta- 2 \alpha} \partial_1 \theta_t+ \\
&+& 
u\cdot \nabla \Lambda^{\beta- 2 \alpha} \partial_1 \theta= [R_{\alpha}, u\cdot \nabla] \theta
\end{eqnarray*}
which gives
\[
f_t+ u . \nabla f+ \Lambda^{\alpha} f+  (-\Lambda^{2 (\beta- \alpha)} \partial_1 \theta- \Lambda^{\beta- 2 \alpha} \partial_1 (u \cdot \nabla \theta)) + u. \nabla \Lambda^{\beta- 2 \alpha} \partial_1 \theta\\
= [R_{\alpha}, u \cdot  \nabla] \theta,
\]
hence
\begin{equation}\label{W1} 
f_t+ u \cdot \nabla f+ \Lambda^{\alpha} f=  \Lambda^{2 (\beta- \alpha)} \partial_1 \theta+ [R_{\alpha}, u\cdot  \nabla] \theta+[\Lambda^{\beta- 2 \alpha} \partial_1, u \cdot  \nabla ] \theta. 
\end{equation}
{\it Note that since $\be-\al=1-2\al<0$,  the term $[\Lambda^{\beta- 2 \alpha} \partial_1, u \cdot  \nabla ] \theta= [\Lambda^{\beta-\al} R_{\al}, u \cdot  \nabla ] \theta$  will always be   easier to treat than the similar term $[R_{\alpha}, u\cdot  \nabla] \theta$. For this reason, we will ignore this term in our discussion, with the understanding that a rigorous proof  can always be produced by following the corresponding  proof for the (harder to treat) commutator term $[R_{\alpha}, u\cdot  \nabla] \theta$.} 

 Based on the definition above
\[
f= G- \Lambda^{ \beta- 2\alpha} \partial_1 \theta= G- R_{\alpha} \Lambda^{\beta- \alpha} \theta= \omega - R_{\alpha} \theta- R_{\alpha} \Lambda^{ \beta- \alpha} \theta= \omega - (R_{\alpha} + R_{\alpha} \Lambda^{ \beta- \alpha}) \theta,
\]
therefore
\begin{equation}
\label{po}
u= \nabla^{\perp} \Delta^{-1} \omega= \nabla^{\perp} \Delta^{-1} f + \nabla^{\perp} \Delta^{-1} R_{\alpha} (I+ \Lambda^{\beta- \alpha}) \theta:= u_f+ u_{\theta}.
\end{equation}
With this definition it is clear that,    $u_f \sim \Lambda^{-1} f$ and $u_{\theta} \sim \Lambda^{- \alpha} \theta+ \Lambda^{1- 3 \alpha} \theta$.\\

\subsection{Regularity criteria for the Boussinesq system}  
\label{sec:1.3}
The question for global regularity is reduced to a certain, so called regularity criteria - namely a quantity\footnote{usually a norm of the solution}, which if controlled up to time $T$, will keep all higher Sobolev norms finite and non-blowing up to time $T$, hence the global regularity. 
  This is a well-known problem in many quasilinear problems, for example in the standard Navier-Stokes posed on \\ $\rone_+\times \rd$, it suffices to control {\it a priori} $\sup_{0\leq t\leq T} \|u(t)\|_{\dot{H}^{d/2}}$ or 
  $\sup_{0\leq t\leq T} \|u(t)\|_{L^d}$ or some mixed norm quantities of the form $\|u\|_{L^p_t(0,T) L^q(\rd)}, \f{2}{p}+\f{d}{q}=1, 2\leq p\leq \infty$. These are all quantities, which of course scale nicely according to the natural scaling of the NLS problem. One difficulty with \ref{BQE} is that the problem does not have scaling invariance, outside of the case $\al=\be$. Nevertheless, there exists a regularity result for the Boussinesq system, namely Theorem 1.2 in \cite{JMWZ}. Although, it is not quite stated in the clean form that we described above for NLS, it provides for a regularity result for the temperature equation\footnote{Given the form of the equation \ref{po}, the motivation for the form of $u$ below  becomes clear} in \ref{BQE}. More precisely, we have 
  \begin{proposition}(Theorem 1.2 in \cite{JMWZ})
  \label{prop:lio}
  Let $\be\in (0,1)$, $\tilde{u}: \nabla\cdot \tilde{u}=0$ with 
  \begin{equation}
  \label{op}
  M=\|\tilde{u}\|_{L^\infty(0,T) L^2(\rtwo)}+ \|\nabla \tilde{u}\|_{L^\infty(0,T) L^\infty(\rtwo)}<\infty.
  \end{equation}
  Assume that $\theta: \theta\in L^2(\rtwo), \nabla \theta\in L^2 \cap L^\infty $ satisfies the  generalized critical surface quasi-geostrophic equation 
  \begin{equation}
  \label{temp}
  \left| \begin{array}{l}
  \theta_t+\La^\be \theta+u\cdot \nabla \theta=0 \\
  u=\tilde{u}+v, v= -\nabla^\perp \La^{-2+\be}\p_1 \theta \\
  \theta(0,x)=\theta_0(x).
  \end{array}
  \right.
  \end{equation}
 Then, the \ref{temp} has an unique solution $\theta \in C([0,T), H^1(\rtwo))$,
 $$
 \|\nabla \theta\|_{L^\infty(0,T) L^\infty(\rtwo)}\leq C(T,M, \|\nabla_0\|_{H^1}, \|\nabla \theta_0\|_{L^\infty}). 
 $$
for some continuous function $C$. 
  \end{proposition}
 Having $\theta$ as smooth as guaranteed by Proposition \ref{prop:lio}  in turn allows us to conclude  the regularity of $u$ in the full Boussinesq system \ref{BQE}. Thus,   the regularity criteria,  which we need,  is exactly 
 $$
  M_T=\sup_{0\leq t\leq T} [\|u_f\|_{L^\infty(0,t)L^\infty_x(\rtwo)}+ \|\nabla u_f\|_{L^\infty(0,t)L^\infty_x(\rtwo)}]<\infty.
  $$
 In order to extract an easy to verify quantity, we make use of the following result.
 \begin{proposition}(Lemma 2.5, p. 1969, \cite{WXXY})
 \label{prop:WXXY}
 Let $\al, \beta: \al+\be\leq 1, \f{1}{2}<\al<1$ and 
 \begin{equation}
 \label{1014}
 G_t+u\cdot \nabla G+\La^\al G=[R_\al, u\cdot] \nabla +\La^{\be-\al} \p_{x_1}\theta.
 \end{equation}
 Then, if $ \f{2}{1-\al}>q>\f{2}{\al}$ and 
 $$
 \sup_{0\leq t\leq T} \|G(t, \cdot)\|_{L^q(\rtwo)}<\infty,
 $$
 then for any $0\geq s<\max(3\al-2,0)$, one has the bound 
 $$
 \sup_{0\leq t\leq T} \|G(t, \cdot)\|_{B^{s}_{r,\infty}}<\infty,
 $$
 where 
 $$
 \f{2}{2\al-1}<r\leq \f{2q}{2-(1-\al)q}.
 $$
 \end{proposition}
Let us mention that the equation $G$ displayed in \ref{1014} corresponds to  the change of variables used in previous works (dubbed first generation Hmidi-Keraani-Rousset). On the other hand, we would like to apply Proposition \ref{prop:WXXY} to the solution $f$ of \ref{W1}. Note however that the terms in \ref{W1} are either the same or more regular than the corresponding terms\footnote{Thanks to Prof. Ye for pointing this out to us in a private communication} in \ref{1015}. Thus, we can apply Proposition \ref{prop:WXXY} to $f$. Using this result, we can reduce matters to verifying 
 \begin{equation}
 \label{1015}
 \sup_{0\leq t\leq T} \|f(t, \cdot)\|_{L^6(\rtwo)}<\infty.
 \end{equation}
 Indeed, assuming that we have established the bound \ref{1015}, we apply Proposition \ref{prop:WXXY} with $q=6$ (which is exactly in the range 
 $(\f{2}{\al}, \f{2}{1-\al})$). We obtain the following bound for $f$ 
 $$
 \sup_{0\leq t\leq T} \|f\|_{B^{3\al-2}_{\f{6}{3\al-2}, \infty}}<\infty.
 $$
 But then, by elementary Sobolev embedding, we have for every small $\de>0$, 
 $$
 \|\nabla u_f\|_{L^\infty_x} \leq C_\de \|f\|_{W^{\f{3\al-2-\de}{3}, \f{6}{3\al-2}}}\leq C_\de \|f\|_{B^{3\al-2}_{\f{6}{3\al-2}, \infty}}
 $$
 which would have verified the bound \ref{op}. 
 Thus, it remains to verify \ref{1015}. 
  
 {\bf Remark:} Originally, our proof proceeded  via a Sobolev embedding control of the form 
 $\|\nabla u_f\|_{L^\infty_x(\rtwo)}\leq C(\|\La^{\de} \nabla f\|_{L^2(\rtwo)}+\|f\|_{L^2})$ and then controlling this last Sobolev norm. We gratefully acknowledge Professor Ye's contribution, which lead us to this much shorter argument.

 \subsection{Strategy of the proof and the organization of the paper}
As we have alluded to before, the strategy is to  follow the standard approach for such models - namely one starts with a local solution\footnote{which is immediately smooth for any time $t>0$}. Such solution may of course be defined for short time only and it may blow up at some finite time $T_0<\infty$. We henceforth do not worry about the existence and the regularity of the solution up to time $T_0$, but we need good {\it a priori} estimates. More precisely, in the discussion leading to  \ref{1015}, we explained that blow up is possible, only if $\limsup_{t\to T_0-} \|  f\|_{L^\infty(0,t)L^6_x(\rtwo)}=\infty$. Thus, a contradiction will be reached (whence $T_0=\infty$ and the solution is global), if one can provide {\it a priori} bound in the form $\sup_{0\leq t\leq T_0}\|\ f\|_{L^\infty(0,t)L^6_x(\rtwo)}=M_0<\infty$. In practice, we construct  $M=M(T; \|\theta_0\|_{L^1\cap H^{2+\rho}(\rtwo)}, \|u_0\|_{H^{1+\rho}(\rtwo)})$ a continuous function in all arguments, so that 
$\sup_{0\leq t\leq T}\| f\|_{L^\infty(0,t)L^6_x(\rtwo)}\leq M(T)$. 

Starting with the obvious {\it a priori} bounds \ref{fazel2}, we gradually improve it to finally obtain \ref{1015}. More precisely, in Section \ref{sec:3}, we first establish an  $L^2$ bound for $f$ (see Proposition \ref{fazel61}), together with some   Sobolev bounds 
for $\theta$. Next, using the $L^2$ bounds from Proposition \ref{fazel61}, we bootstrap Proposition \ref{propo:3}, in order to establish  $L^4$ bounds for $f$, 
together with uniform in time Sobolev bounds for $f, \theta$ and some $L^2$ averaged in time Sobolev bounds. These are all (considerably) better than the one in Proposition \ref{fazel61}. We finish Section \ref{sec:3} by bootstrapping Proposition \ref{propo:3} yet again to establish $L^6$ bounds for $f$, together with even better uniform and $L^2$ time averaged Sobolev bounds for $f, \theta$. The uniform in time Sobolev bounds required for the global regularity 
in \ref{1015} do not come cheaply and by themselves - instead one seems to need to cook up energy functionals involving $L^p$ ($p$ larger) norms of $f$. In other words, for low $\al$ one faces not only the usual derivative difficulties as in previous works, but also integrability issues for $f$. 
Having Proposition \ref{propo:4} is enough, by the discussion in Section \ref{sec:1.3} below to conclude the global regularity claimed in Theorem \ref{theo:1}.

\section{Preliminaries}
\label{sec:2}
\noindent  For the proof, we need a number of technical tools, which we now introduce. We start with the $L^p$ spaces and Littlewood-Paley theory. 
\subsection{Function spaces} 
We use standard notation for $L^p$ spaces and Sobolev spaces, namely for $s>0, p\in [1,\infty)$, 
\begin{eqnarray*}
\|f\|_{L^p} &=& \left( \int |f(x)|^p dx\right)^{1/p} \\
\|f\|_{W^{s,p}} &=& \|\La^{s} f\|_{L^p}+ \|f\|_{L^p}
\end{eqnarray*}
We need to quickly introduce some elementary Littlewood-Paley theory. To that end, let $\Upsilon$ be an even and  smooth function on $\rone$, so that $supp \ \Upsilon\subset [-2,2]$, so that $\Upsilon(\xi)=1, |\xi|<1$. Define $\zeta:\rtwo\to \rone$ via $\zeta(\xi)=\Upsilon(|\xi|)-\Upsilon(2|\xi|)$, so that $\zeta\in C^\infty(\rtwo)$, with $supp \ \zeta\subset \{\xi: \f{1}{2}<|\xi|<2\}$. In addition, 
$$
\sum_{k=-\infty}^\infty \zeta(2^{-k} \xi)=1, \ \ \xi\neq 0.
$$
This allows us to define the  Littlewood-Paley operators  $\widehat{ \Delta_j  f}(\xi):=\zeta(2^{-j} \xi) \hat{f}(\xi)$,  restricting the Fourier transform of $f$ to the annulus $\{\xi:  |\xi|\sim 2^j \}$.   
We will often denote $f_k=\De_k f$, 
$f_{\sim k}=\sum_{j=k-10}^{k+10} \De_{j}f$ and $f_{<k}=\sum_{j<k} \De_j f$. 

\subsection{Commutator estimates}
\label{sec:2.2}
In this section, we present some commutator estimates, which will be useful in our arguments. Some of them, Lemma \ref{le:com1} and Lemma \ref{le:com2} appear to be new. We start with a lemma developed in \cite{SW} (see Lemma 2.5 there and Corollary 2.6).  
\begin{lemma}
\label{le:com} 
Let $\nabla \cdot g= 0$, $0 < S <1$ and $1 < p_2 < \infty$, $1 < p_1, p_3 \leq \infty$, so that $\frac{1}{p_1}+ \frac{1}{p_2}+ \frac{1}{p_3} = 1$. For every $0 \leq S_1, S_2, S_3 \leq 1$ that satisfy $S_1+ S_2+ S_3 > 1+S $, there exists a $C= C (p_1, p_2, p_3, S_1, S_2)$, so that
\begin{equation}\label{aaa}
| \int_{R^d} h [\La^S, g \cdot  \nabla] \psi dx| \leq C \|\Lambda^{S_1} \psi\|_{L^{p_1}} \|\La^{S_2} h\|_{L^{p_2}} \|\Lambda^{S_3} g\|_{p_3}.
\end{equation}
In particular if $p_3 < \infty$ then
\begin{itemize}
\item  for $S_1= s_1$, $S_2= s_2$ and $S_3= 1$ where $s_1+ s_2 > 1- \alpha$
\begin{equation}{\label{fazel5}}
| \int_{R^d} h [R_{\alpha}, V \cdot \nabla] \psi dx| \leq C \|\Lambda^{s_1} \psi\|_{L^{p_1}} \|\La^{s_2} h\|_{L^{p_2}} \|\nabla V\|_{L^{p_3}},
\end{equation}
 \item  
 similarly, for every $0\leq s_2, s_3<1$, so that  $s_2+ s_3 > 1+S$ we have 
\begin{equation}{\label{fazel6}}
| \int_{R^d} h [\La^S, V \cdot  \nabla] \psi dx| \leq C \| \psi\|_{L^{p_1}} \|\La^{s_2} h\|_{L^{p_2}} 
\|\Lambda^{s_3} V\|_{L^{p_3}}.
\end{equation}
\end{itemize}
Note that in all statements, one could have replaced $R_\al=\p_1 \La^{-\al}$ by any multiplier, which acts as differentiation of order $1-\al$, for example $\La^{1-\al}$. 
\end{lemma}
Note that in this lemma, one has to always allow for small derivative loss. Lemma \ref{le:com} will be adequate for many terms, except when we need to account for all derivatives. In other words, we need a variant which is lossless in the derivative count (and/or endpoint estimates).   We have two versions - 
 Lemma \ref{le:com1} is  for estimates in (homogeneous) Sobolev spaces of negative index, and the other one, Lemma \ref{le:com2}, for estimates in (homogeneous) Sobolev spaces of positive index.  We mostly need Lemma \ref{le:com1} throughout the paper, the need for Lemma \ref{le:com2} arises at the very end of our argument. Interestingly, in the proof (presented in the Appendix), we do not distinguish much between the two cases. Note that the results in Lemmas \ref{le:com1} and Lemma 
 \ref{le:com2} hold under somewhat more general assumptions that the one that we displayed below, but we prefer to keep it simple and convenient for the applications. 
\begin{lemma}
\label{le:com1}
Let $s_1, s_2$ be two reals so that $0\leq s_1$ and  $0\leq s_2-s_1\leq 1$. Let  $p,q,r$ be related via the H\"older's  $\f{1}{p}=\f{1}{q}+\f{1}{r}$, where 
$2<q<\infty$, $1<p,r<\infty$. Finally, let $\nabla \cdot V=0$.   

Then for any $a\in [s_2-s_1, 1]$
\begin{eqnarray}
\label{20}
& & \|\La^{-s_1}[\La^{s_2}, V \cdot \nabla] \vp]\|_{L^p}\leq C \|\La^a V\|_{L^q} \|\La^{s_2-s_1+1-a}\vp\|_{L^r}
\end{eqnarray}
In addition, we have the following end-point estimate. For $s_1>0, s_2>0, s_3>0$ and $s_1<1, s_3<1, s_2<s_1+s_3$, there is\footnote{Note that in the statement of \ref{25}, one does not necessarily need precisely the form $\La^{-s_3} V$. In fact, the estimate applies for any Fourier multiplier $Q$, with the property that 
$\| Q V_{k}\|_{L^\infty}\sim 2^{-k s_3} \|V_k\|_{L^\infty}$}
\begin{equation}
\label{25}
\|\La^{-s_1}[\La^{s_2}, \La^{-s_3} V\cdot \nabla]\vp\|_{L^2}\leq C \|V\|_{L^\infty} \|\La^{s_2-s_1+1-s_3}\vp\|_{L^2}. 
\end{equation}
\end{lemma}

We have the following useful corollary of \ref{20}. 
\begin{corollary}
\label{cor:lk}
Let $p_1, p_2, p_3: \f{1}{p_1}+\f{1}{p_2}+\f{1}{p_3}=1$ and $p_1>2$. Assume that $0\leq s\leq 1$. Then,  
\begin{eqnarray}
\label{f:10} 
& & | \dpr{[\La^{s}, V \cdot \nabla] \vp]}{\psi}| \leq C \|\nabla V\|_{L^{p_1}} \|\La^{s}\vp\|_{L^{p_2}} \| \psi\|_{L^{p_3}}\\
\label{f:20}
& & |\dpr{[\La^{s}, V \cdot \nabla] \vp]}{\psi}| \leq C 
\|\La^a V\|_{L^{p_1}} \|\La^{s+1-a}\vp\|_{L^{p_2}} \| \psi\|_{L^{p_3}}
\end{eqnarray}
whenever $a\in [s,1]$. 
\end{corollary}
  The next lemma is basically identical to Lemma \ref{le:com1}, except that $s_1$ has the opposite sign. 
 \begin{lemma}
\label{le:com2}
Let $0 \leq s_1$, $0<s_2$, $0 \leq s_1+ s_2< 1$, $s_2+ s_1 < a \leq 1$,  $2<q<\infty, 1<r<\infty$ and $\f{1}{p}= \f{1}{q}+ \f{1}{r}$. Then
\begin{eqnarray}
\label{201}
& & \|\La^{s_1}[\La^{s_2}, V \cdot \nabla] \vp]\|_{L^p}\leq C \|\La^a V\|_{L^q} \|\La^{1+ s_2+ s_1- a}\vp\|_{L^r}
\end{eqnarray}
\end{lemma}

The following corollary is a direct result of the above lemma
\begin{corollary}
Let $0 \leq s < \al$, $\be+ s < a \leq 1$, $2<q, r<\infty$ and $\f{1}{2}= \f{1}{q}+ \f{1}{r}$. Then
\begin{eqnarray}
\label{200}
& & \|\La^{s}[R_{\al}, V \cdot \nabla] \vp]\|_{L^2}\leq C \|\La^a V\|_{L^q} \|\La^{1+ \be+ s- a}\vp\|_{L^r}
\end{eqnarray}
\end{corollary}  

Next, we need to prepare a   technical point, which will be useful in the sequel.  
\subsection{The scaled variables} 
For technical reasons, we use the following scaled variables 
\begin{eqnarray}{\label{fazel6'}}
\left\{
  \begin{array}{ll}
  \theta(t, x)= \Theta (\epsilon_0^{\beta} t, \epsilon_0 x);~~ u(t, x)= U (\epsilon_0^{\beta} t, \epsilon_0 x)\\
    f(t, x)= F (\epsilon_0^{\beta} t, \epsilon_0 x)~~~,~~ U= U_F+ U_{\Theta},
  \end{array}
\right.
\end{eqnarray}
where $\epsilon_0$ is a small parameter to be determined in each energy estimate later on separately. Clearly  
\[
\Theta_t+ \epsilon_0^{\alpha} U . \nabla \Theta+ \Lambda^{\beta} \Theta  = 0
\]
The corresponding equation for $F$ is 
\begin{eqnarray*}
\epsilon_0^{\beta} F_t+ \epsilon_0 U\cdot \nabla F+ \epsilon_0^{\alpha} \La^{\al} F &=&  \epsilon_0^{1+ 2(\beta- \alpha)} \Lambda^{2( \beta- \alpha)} \partial_1 \Theta+ \epsilon_0^{3 \beta} [\Lambda^{\beta- 2
 \alpha} \partial_1, U\cdot  \nabla] \Theta+ \\
 &+& \epsilon_0^{1+ \beta} [R_{\alpha}, U\cdot\nabla] \Theta
\end{eqnarray*}
Thus, our new system now is in the form of
\begin{equation}{\label{fazel6''}}
\left\{
  \begin{array}{ll}
F_t+ \epsilon_0^{\alpha} U. \nabla F+ \epsilon_0^{\alpha- \beta} \Lambda^{\alpha} F =  N(U, F, \Theta), \\
 \Theta_t+ \epsilon_0^{\alpha} U . \nabla \Theta+ \Lambda^{\beta} \Theta  =  0.
 \end{array}
\right.
\end{equation}
with $N(U, F, \Theta)=\epsilon_0^{2- 3 \alpha}\Lambda^{2( \beta- \alpha)} \partial_1 \Theta+ \epsilon_0  [R_{\alpha}, U. \nabla] \Theta 
+\epsilon_0^{2 \beta}  [\Lambda^{\beta- 2
 \alpha} \partial_1, U. \nabla] \Theta$. 
Note that in this case
$\|\theta\|_{L^p}=\epsilon_0^{-2/p} \|\Theta\|_{L^p}$, in particular $\|\Theta\|_{L^{\infty}} = \|\theta\|_{L^{\infty}}$ and similar for $f,F$. 
\subsection{Some basic energy inequalities}  
Now suppose $\kappa, s \geq 0$ , and take $\La^s$ and $\La^{\kappa}$ derivatives,  and then dot product with $\La^s F$ and $\La^{\kappa} \Theta$ in \ref{fazel6''}, respectively, to get
\begin{eqnarray}\label{641}
\frac{1}{2} \partial_t \|\La^s F\|^2_{L^2} &+& \epsilon_0^{\alpha- \beta}\|\La^{s+\frac{\alpha}{2}} F\|^2_{L^2} \leq \epsilon_0^{\alpha} | \int (\La^s [U . \nabla F]) . \La^s F dx| \\ \nonumber
&+& \epsilon_0^{2- 3\alpha} |\langle \La^{2 (\beta- \alpha)+ s} \partial_1 \Theta, \La^s F\rangle| +
\epsilon_0 |\langle \La^s [R_{\alpha},  U . \nabla] \Theta, \La^s F\rangle|\\ \nonumber
&+&  \epsilon_0^{2 \beta} |\langle \La^s [\La^{\beta- 2 \alpha} \partial_1,  U . \nabla] \Theta, \La^s F\rangle| = I_1+ I_2+ I_3+ I_4
\end{eqnarray}
and,
\begin{equation}\label{65}
\frac{1}{2} \partial_t \|\La^{\kappa} \Theta\|^2_{L^2} + \|\La^{\kappa+\frac{\beta}{2}} \Theta\|^2_{L^2} \leq \epsilon_0^{\alpha} |\langle \La^{\kappa} (U \cdot \nabla \Theta), \La^{\kappa} \Theta\rangle|:= I_5.
\end{equation}
In the case that $s< 1$ or $\kappa < 1$ we can easily rewrite $I_1$ and $I_5$ in the commutator forms:
\[
I_1= \epsilon_0^{\alpha} |\langle [\La^s, U . \nabla] F, \La^s F\rangle|,~~~~~~~I_5= \epsilon_0^{\alpha} |\langle [\La^{\kappa}, U . \nabla] \Theta, \La^{\kappa} \Theta\rangle|.
\]
Now, take  dot product with $F |F|^{p-2}$ in \ref{fazel6''}, and get 
\begin{eqnarray*}
\frac{1}{p} \partial_t \|F \|^p_{L^p}&+& \epsilon_0^{\alpha- \beta} |\int F |F|^{p-2} \Lambda^{\alpha} F dx| \leq \epsilon_0^{2- 3\alpha} \int F |F|^{p-2} \Lambda^{2(\beta- \alpha)} \partial_1 \Theta dx\\
&+& \epsilon_0 |\langle  [R_{\alpha}, U. \nabla] \Theta, F |F|^{p-2}\rangle| + \epsilon_0^{2 \beta} |\langle  [\Lambda^{\beta- 2 \alpha} \partial_1, U. \nabla] \Theta, F |F|^{p-2} \rangle|\\
&:=& K_1+ K_2+ K_3.
\end{eqnarray*}
By maximum principle 
\begin{eqnarray*}
 \epsilon_0^{\alpha- \beta} |\int F |F|^{p-2}  \Lambda^{\alpha} F dx|  &\geq &  C_0 \epsilon_0^{2\alpha- 1} \int |\Lambda^{\frac{\alpha}{2}} (F^{\frac{p}{2}})|^2 dx \geq C_0 \epsilon_0^{2\alpha- 1} \|F^{\frac{p}{2}}\|^2_{L^{\frac{4}{2- \alpha}}}= \\
 &=& C_0 \epsilon_0^{2\alpha- 1} \|F\|^p_{L^{\frac{2 p}{2- \alpha}}}.
\end{eqnarray*} 
Therefore
\begin{eqnarray}\label{21}
\frac{1}{p} \partial_t \|F \|^p_{L^p}&+& C_0 \epsilon_0^{\al- \be} \|F\|^p_{\frac{2 p}{2- \alpha}} \leq \epsilon_0^{2- 3\alpha} \int F^3 \Lambda^{2(\beta- \alpha)} \partial_1 \Theta dx = K_1+ K_2+ K_3.
\end{eqnarray}
In our  proofs, we usually combine two or three relations of \ref{641}, \ref{65} and \ref{21}, with different $\kappa$, $s$, and $p$, and try to find the proper estimate for the right hand side, and then use the Gronwall's inequality to close the arguments. In our discussion, we shall ignore the estimates for $I_4$ and $K_3$, as they are easier to deal with than the corresponding terms $I_3$ and $K_2$.

\section{$L^p$ bounds on $f$}
\label{sec:3} 
In this section we prove  $L^2$, $L^4$ and $L^6$ bound for $f$. We start with $L^2$ bound and then proceed with $L^4$  bound and finally we get the $L^6$ bound. During the discussion we also raise the derivative on both $\theta$ and $f$. This allows us to jump to higher derivatives in the next sectin. \\
\subsection{$L^2$ Estimate}~
\begin{proposition}{\label{fazel61}}
Let $0 < \rho << 1$, $\gamma= \frac{\beta}{2}- 2 \rho$, $f_0 \in H^{\frac{\alpha}{2}}$ and $\theta_0 \in L^{\infty} \cap H^{\gamma}$ then
\begin{eqnarray}
\label{g:10}
  \|f\|_{L^2} + \|\Lambda^{\gamma} \theta\|_{L^2} &\leq & C_T\\
  \label{g:20}
  \int_0^T (\|\Lambda^{\frac{\alpha}{2}}f (., t)\|_{L^2}^2 &+& \|\Lambda^{\frac{\beta}{2}+ \gamma} \theta   (., t)\|_{L^2}^2 ) dt \leq  C_T
\end{eqnarray}
where $C_T= C(T, \|\theta_0\|_{L^{\infty}}, \|f\|_{H^{\frac{\alpha}{2}}}, \|\theta\|_{H^{\frac{\alpha}{2}}})$.
\end{proposition}

\begin{proof}(Proposition \ref{fazel61}) 
We start with the scaled variables. In each case, we specify how small $\epsilon$ needs to be in order to close the estimates. In the end, we choose and fix one such $\epsilon$, say the half of the smallest upper bound. This argument will then imply the estimates \ref{g:10} and \ref{g:20}. 

In \ref{641} and \ref{65} take $\kappa= 0$ and $s= \ga$, then we want to bound the right hand side of the following relation
\begin{equation}\label{1000}
\f{1}{2} \partial_t (\|F\|^2_{L^2}+ \|\La^{\ga} \Theta\|^2_{L^2})+ \epsilon_0^{\al- \be} \|\La^{ \f{\al}{2}} F\|_{L^2}+ \|\La^{\ga+ \f{\be}{2}} \Theta\|_{L^2} \leq I_1+ I_2+ I_3+ I_4+ I_5
\end{equation}
Since $\langle U\cdot \nabla F, F \rangle=-\f{1}{2} \dpr{\nabla\cdot U}{F^2}=0$, we have  $I_1= 0$.

\subsubsection{Estimate for $I_2$}~\\
{\bf Case $1,\al>\f{3}{4}$}:
$$
I_2\leq \epsilon_0^{2- 3 \alpha} \|\La^{3-4\al}\Theta\|_{L^2} \|F\|_{L^2}\leq \epsilon_0^{2- 3 \alpha} \|\Theta\|_{L^{\f{1}{2\al-1}}} \|F\|_{L^2} \leq 
\frac{1}{100} \|F\|_{L^2}^2 + C_{\epsilon_0}.
$$
{\bf Case $2,\al \leq \f{3}{4}$ }:\\
 we have by H\"older's and Gagliardo-Nirenberg, 
\begin{eqnarray*} 
I_2 
\leq \epsilon_0^{2- 3 \alpha} \|\Lambda^{3- 4 \alpha} F\|_{L^2} \| \Theta\|_{L^2} \leq C \epsilon_0^{2- 3 \alpha} \|\theta_0\|_{L^2} \|\La^{\f{\al}{2}} F\|_{L^2}^{\de} \|F\|_{L^2}^{1-\de}.
\end{eqnarray*}
for $\de=\f{3-4\al}{\al/2}$.  Note that $\de\in (0,1)$, since $\al>2/3$. Applying Young's inequality gives us 
$$
 I_2 \leq \f{\epsilon_0^{2\al-1}}{100} \|\La^{\f{\al}{2}} F\|_{L^2}^2+ C_{\epsilon_0, \|\theta_0\|_{L^2}} (1+\|F\|_{L^2})^2. 
$$
\subsubsection{Estimate for $I_3$}~\\
In this case we are seeking   bounds for two terms $I_3^f$ and $I_3^{\theta}$
\[
I_3 \leq \epsilon_0  |\int F [R_{\alpha}, U_{\Theta} . \nabla] \Theta dx|+ \epsilon_0 |\int F [R_{\alpha}, U_F . \nabla] \Theta dx| :=  I_3^{\theta}+ I_3^f
\]
 Now, for $I_3^{\theta}$, we apply \ref{fazel6}, with $p_1=\infty, p_2=p_3=2$ and $s_3=\f{\be}{2}+\ga+\al=1-2\rho$ and 
$s_2=\f{\al}{2}$. Note that this is within the range of applicability of \ref{fazel6}, since  $s_2+s_3=1+\f{\al}{2}-2\rho>2-\al$, whenever  $\al>\f{2}{3}$ and $0<\rho<<\al-\f{2}{3}$. We get 
$$
I_3^{\theta}\leq C \epsilon_0 \|\Theta\|_{L^\infty} \|\La^{\f{\al}{2}} F\|_{L^2} \|\La^{\f{\be}{2}+\ga+\al} U_\Theta\|_{L^2}\leq 
C\|\theta_0\|_{L^\infty} \|\La^{\f{\al}{2}} F\|_{L^2}  \|\La^{\f{\be}{2}+\ga} \Theta\|_{L^2}
$$
Thus, 
$$
I_3^{\theta}\leq    \f{\epsilon_0^{2\al-1}}{100}  \|\Lambda^{\frac{\alpha}{2}} F\|_{L^2}^2+
C \epsilon_0^{3-2\al}  \|\theta_0\|_{L^\infty}^2 \|\La^{\f{\be}{2}+\ga} \Theta\|_{L^2}^2
$$
Taking $\epsilon_0: C \epsilon_0^{3-2\al}  \|\theta_0\|_{L^\infty}^2\leq \f{1}{100}$ will ensure that we can absorb the second term above behind $\|\La^{\f{\be}{2}+\ga} \Theta\|_{L^2}^2$ on the left-hand side. 

Regarding $I_3^f$, we have by \ref{fazel5} with $p_3=2$, $s_1=0$, $s_2=1-\al+\rho$, 
$
\f{2}{p_1}=\f{3\al}{2}-1-\rho, \f{2}{p_2}=2-\f{3\al}{2}+\rho.
$
\begin{eqnarray*}
I_3^f &\leq & \epsilon_0 \|\nabla U_F\|_{L^2} \|\La^{1-\al+\rho} F\|_{L^{p_2}}\|\Theta\|_{L^{p_1}}\leq 
\epsilon_0 C  \|F\|_{L^2} \|\La^{\al/2} F\|_{L^2}\|\theta_0\|_{L^{p_1}}.
\end{eqnarray*}
where we have used the Sobolev embedding estimate $\|\La^{1-\al+\rho} F\|_{L^{p_2}}\leq C \|\La^{\al/2} F\|_{L^2}$. Applying Cauchy-Schwartz yields 
\begin{eqnarray*}
I_3^f &\leq & \f{\epsilon_0^{\al-\be}}{100} \|\La^{\al/2} F\|_{L^2}^2+ \epsilon^{1+2 \be} C \|\theta_0\|_{L^{p_1}}^2  \|F\|_{L^2}^2\\
&\leq & \f{\epsilon_0^{\al-\be}}{100} \|\La^{\al/2} F\|_{L^2}^2+ \f{1}{100} \|\theta_0\|_{L^{p_1}}^2  \|F\|_{L^2}^2,
\end{eqnarray*}
where we took $\epsilon$ so that $\epsilon^{1+2 \be} C \leq \f{1}{100}$.
\subsubsection{Estimate for $I_5$}~\\
For $I_5^f$, take $s_3=1-\rho, s_2=\ga+2\rho$, $p: \f{1}{p}=\f{1}{2}-\f{\rho}{2}$, $q: \f{1}{q}=\f{\rho}{2}$, then we have by \ref{fazel6},  
$$
|\langle[\Lambda^{\gamma}, U_F . \nabla] \Theta, \Lambda^{\gamma} \Theta \rangle|\leq C \|\Theta\|_{L^q} 
\|\La^{2\ga+2\rho} \Theta\|_{L^2} \|\La^{1-\rho} U_F\|_{L^p}.
$$
Also, by Sobolev embedding $\|\La^{1-\rho} U_F\|_{L^p}\leq C \|\La^{-\rho} F\|_{L^p}\leq C\|F\|_{L^2}$. All in all, noting that $2(\ga+\rho)=\ga+\be/2$, 
\begin{eqnarray*}
I_5^f= \epsilon_0^{\al} |\langle[\Lambda^{\gamma}, U_F . \nabla] \Theta, \Lambda^{\gamma} \Theta \rangle| &\leq& 
\f{1}{100} \|F\|_{L^2}^2+ \epsilon_0^{2 \al} C \|\Theta_0\|_{L^q}^2 \|\La^{\be/2+\ga} \Theta\|_{L^2}^2\\ &\leq&   
\f{1}{100} \|F\|_{L^2}^2+ \f{\epsilon^{\al- \be}}{100}  \|\La^{\be/2+\ga} \Theta\|_{L^2}^2. 
\end{eqnarray*}
where we took $\epsilon_0$ so that $\epsilon_0 C \|\theta_0\|_{L^{q}} \leq \f{1}{100}$.\\
For the term containing $U_\Theta$, we have by \ref{25}, with $s_1=\be/2, s_2=\ga, s_3=\al$, 
\begin{equation}
\label{g:80}
 I_5^{\theta}= \epsilon_0^{\al} |\langle[\Lambda^{\gamma}, U_\Theta\cdot  \nabla] \Theta, \Lambda^{\gamma} \Theta \rangle|\leq C \epsilon_0^{\al}\|\theta_0\|_{L^\infty} \|\La^{\ga+\be/2} \Theta\|_{L^2}^2 \leq \f{1}{100} \|\La^{\ga+\be/2} \Theta\|_{L^2}^2
\end{equation}
where we take $C \epsilon_0^{\al} \|\theta_0\|_{L^\infty} \leq \f{1}{100}$. 
 Introducing 
$$
J(t)= \|\La^\ga \Theta\|_{L^2}^2+ \|F\|_{L^2}^2,
$$
and putting all the estimates together, we obtain the bound 
$$
J'(t)+ \|\Lambda^{\frac{\alpha}{2}}f\|_{L^2}^2+\|\Lambda^{\frac{\beta}{2}+ \gamma} \theta \|_{L^2}^2\leq C_{\epsilon_0, \|\theta_0\|_{L^2\cap L^\infty}} J(t)
$$
An application of the Gronwall's inequality yields the bounds for the right hand side of \ref{1000}.
\end{proof}
Now that we have the estimate for $\sup_{0\leq t\leq T} \|f\|_{L^2}$, we use it to obtain the estimates for $\sup_{0\leq t\leq T} \|f\|_{L^4}$. 
\subsection{$L^4$ Estimate}
The precise result that we prove is the following. 
\begin{proposition}
\label{propo:3}
Let $1> \alpha > \frac{2}{3}$, $(u, \theta)$ be the solution of \ref{BQE} and $(u_0, \theta_0)$ be as specified in Theorem $(1)$. Assume $f$ satisfies \ref{W1}, then for any $T> 0$ there exists $C_T= C(T)$, such that
\begin{eqnarray*}
\sup_{0 \leq t \leq T} \|f\|_{L^4} + \int_0^T \|f\|^4_{L^{\frac{8}{2- \alpha}}} & < & C_T\\
 \sup_{0 \leq t \leq T} \|\Lambda^{\frac{\alpha}{2}} f\|_{L^2} + \int_0^T  \|\Lambda^{\alpha} f\|^2_{L_x^2} dt & < & C_T\\
 \sup_{0 \leq t \leq T} \|\Lambda^{ \frac{3\beta}{2}} \theta\|_{L^2} + \int_0^T  \|\Lambda^{2 \beta} \theta\|^2_{L_x^2}) dt & < & C_T\\
\end{eqnarray*}

\end{proposition}
\begin{proof}
We again use the scaled variables. 
In \ref{641}, \ref{65} and \ref{21} take $\kappa= \f{\al}{2}$, $s= \f{3 \be}{2}$ and $p= 4$ to get 
 \begin{eqnarray*}
& &  \partial_t( \frac{1}{4} \|F \|^4_{L^4}+ \frac{1}{2} \|\Lambda^{\frac{3 \beta}{2}} \Theta \|^2_{L^2} +\frac{1}{2} \|\Lambda^{\frac{\alpha}{2}} F \|^2_{L^2})  +  C_0 \epsilon_0^{\al- \be} \|F\|^4_{L^{\frac{8}{2- \alpha}}}+\\
 &+& \epsilon_0^{\al- \be}  \|\Lambda^{\alpha} F \|^2_{L^2}
 +  \|\Lambda^{2 \beta} \Theta \|^2_{L^2}
 \leq K_1+ K_2+ K_3+ I_1+ I_2+ I_3+ I_4
 \end{eqnarray*}
 We now proceed to establish proper bounds for each term in the right hand side.
\subsubsection{Estimate for $K_1$}~\\
{\bf Case $3- 4 \alpha < 0$:}~\\
 In this case we have,
  \[
   |\int F^3  \Lambda^{3- 4 \alpha} \Theta ~dx| \leq \|F\|^3_{L^4} 
   \|\Lambda^{3- 4 \alpha} \Theta\|_{L^4},
  \]
  and by Sobolev embedding 
  \[
  \|\Lambda^{3- 4 \alpha} \Theta\|_{L^4} \leq \|\Theta\|_{L^{\f{1}{2\al-\f{5}{4}}}}\leq C_\epsilon \|\theta_0\|_{L^{\f{1}{2\al-\f{5}{4}}}}.
  \]
 hence 
 $$
  K_1 \leq  
   \|F\|^4_{L^4}+ C_{\epsilon}.
  $$
  \\
{\bf Case $3- 4 \alpha >0$:} 
  We have 
  \[
   |\int F^3 \Lambda^{3- 4 \alpha} \Theta ~dx| \leq \|F\|^3_{L^{\frac{8}{2- \alpha}}} \|\Lambda^{3- 4 \alpha} \Theta\|_{L^{\frac{8}{2+ 3 \alpha}}}.
  \]
Furthermore, 
  \[
  \|\Lambda^{3- 4 \alpha} \Theta\|_{\frac{8}{2+ 3 \alpha}} \leq 
  \|\Lambda^{2 \beta} \Theta\|^a_{L^2} \|\Theta\|^{1- a}_{L^{q_0}}\leq C_{\epsilon_0} \|\Lambda^{2 \beta} \Theta\|^a_{L^2} \|\theta_0\|^{1- a}_{L^{q_0}}
  \] 
  where
  \[
  a= \frac{3- 4 \alpha}{2 \beta}
  \]
  and
  \[
  q_0= \f{4(2 \al- 1)}{3 \al \be+ 6 \al- 4}
  \]
 Note that  for $\al>2/3$, $q_0 \geq 1$. 
 Therefore
 \[
   K_1 \leq C \epsilon_0^{2- 3 \alpha} \|F\|^3_{L^{\frac{8}{2- \alpha}}} \|\Lambda^{2 \beta} \Theta\|^a_{L^2} \leq \frac{\epsilon_0^{\al- \be}}{100} \|F\|^4_{L^{\frac{8}{2- \alpha}}}+ \epsilon_0^{\f{9- 14 \alpha}{2}} C \|\Lambda^{2 \beta} \Theta\|^{4 a}_{L^2}.
\]
Clearly $4 a < 2$, hence
\[
K_1 \leq \frac{\epsilon_0^{\al- \be}}{100} \|F\|^4_{L^{\frac{8}{2- \alpha}}}+ \frac{1}{100} \|\Lambda^{2 \beta} \Theta\|^2_{L^2}+ C_{\epsilon_0}
\]
\subsubsection{Estimate for $K_2$}~\\
{\bf Estimate for $K_2^f$:}\\
For $0 < \delta << 1$, to be determined later, by \ref{f:10} with
$s= 1- \alpha$, then
\begin{eqnarray*}
K_2^f= \epsilon_0 |\langle [R_{\alpha}, U_F. \nabla] \Theta, F^3 \rangle |  &\leq & \epsilon_0 C \|\Theta \|_{L^{\frac{8}{3\alpha- 2}}} \| \Lambda^{1- \alpha} (F^3)\|_{L^{\frac{4}{4- \alpha}}} \|\nabla U_F\|_{L^{\frac{8}{2- \alpha}}} \\
 & \leq & \epsilon_0 C \| \Lambda^{1- \alpha} F\|_{L^2} \|F\|^3_{L^{\frac{8}{2- \alpha}}} \|\Theta \|_{L^{\frac{8}{3\alpha- 2}}},
\end{eqnarray*}
 Hence we conclude 
\[
K_2^f
\leq \frac{\epsilon_0^{\al- \be}}{100} \|F\|^4_{L^{\frac{8}{2- \alpha}}}+ 
C_{\epsilon_0}  \|\Lambda^{1-\al} F \|^4_{L^2}.
\]
But 
\[
\|\Lambda^{1-\al} F \|^4_{L^2} \leq \|\Lambda^{\alpha} F \|^{\frac{4(1-\al)}{\alpha}}_{L^2} 
\|F\|^4_{L^2}.
\]
Note however  that $\frac{4(1-\al)}{\alpha} < 2$, since $\alpha > \frac{2}{3}$, therefore 
\[
K_2^f \leq \frac{\epsilon_0^{\al- \be}}{100} \|F\|^4_{L^{\frac{8}{2- \alpha}}}+ \frac{\epsilon_0^{\al- \be}}{100} \| \Lambda^{\alpha} F \|^2_{L^2}+ 
C_{\epsilon_0}. 
\]
{\bf Estimate for $K_2^{\theta}$:}\\ 
Again for $0< \de<<1$, apply \ref{fazel6} with $s_3= \alpha+ \delta$ and 
$s_2= 2 (1- \alpha)$
\begin{eqnarray*}
K_2^{\theta}= \epsilon_0 |\dpr{[R_{\alpha}, U_{\Theta}. \nabla] \Theta}{F^3}| & \leq & \epsilon_0 \|\Lambda^{\alpha+ \delta} U_{\Theta}\|_{L^{\frac{8}{\alpha}}} \|\Lambda^{2 (1- \alpha)} (F^3)\|_{L^{\frac{4}{4- \alpha}}} \|\Theta\|_{L^{\frac{8}{\alpha}}}\\
&\leq & \epsilon_0 C \|\Lambda^{ \delta} \Theta\|_{L^{\frac{8}{\alpha}}} \|F\|^2_{L^{\frac{8}{2- \alpha}}} \|\Lambda^{2 (1- \alpha)} F\|_{L^2} \|\theta_0\|_{L^{\frac{8}{\alpha}}}.
\end{eqnarray*}
Note 
\[
\|\Lambda^{ \delta} \Theta\|_{L^{\frac{8}{\alpha}}} \leq \|\Lambda^{ \frac{\beta}{2}- 2 \rho} \Theta\|^a_{L^2} \|\theta\|^{1- a}_{L^q}
\]
where $a=\f{\de}{\f{\be}{2}-2\rho}$ and $q:\f{1-a}{q}+\f{a}{2}=\f{\al}{8}$. Clearly $q\in (1, \infty)$, provided $\de<<1$. 
 We have  obtained 
\begin{eqnarray*}
K_2^{\theta} \leq  \frac{\epsilon_0^{\al- \be}}{100} \|F\|^4_{L^{\frac{8}{2- \alpha}}} + \epsilon_0^{1+ 2\be} C
\| \Lambda^{2 (1- \alpha)} F \|^2_{L^2}
\end{eqnarray*}
and
\[
\| \Lambda^{2 (1- \alpha)} F \|_{L^2} \leq \| \Lambda^{\alpha} F \|^a_{L^2} \|F\|^{1- a}_{L^2}
\]
where $a= \frac{2 (1- \alpha)}{\alpha} < 1$. Thus, 
\begin{eqnarray*}
K_2^{\theta} \leq  \frac{\epsilon_0^{\al- \be}}{100} \|F\|^4_{L^{\frac{8}{2- \alpha}}}+ \frac{\epsilon_0^{\al- \be}}{100} \|\Lambda^{\alpha} F \|^2_{L^2}+ C_{\epsilon} 
\end{eqnarray*}
hence
\[
K_2 \leq \frac{\epsilon_0^{\al- \be1}}{50} \|F\|^4_{L^{\frac{8}{2- \alpha}}}+ \frac{\epsilon_0^{\al- \be}}{50} \|\Lambda^{\alpha} F \|^2_{L^2}+ C_{\epsilon} 
\]
\subsubsection{Estimate for $I_1$}~ \\
{\bf Estimate for $I_1^f$:} 
\[
I_1^{f} = \epsilon_0^{\al} |\langle [\La^{\f{\al}{2}}, U_F . \nabla] F, \La^{\f{\al}{2}} F\rangle| \leq \epsilon_0^{\al} \|\La^{\al} F\|_{L^2} \|\La^{- \f{\al}{2}} [\La^{\f{\al}{2}}, U_F . \nabla] F\|_{L^2}
\]
now in \ref{20} take $s_1= s_2= \f{\al}{2}$, $V= \La^{-1} F$, $\vp= F$, $a= 1$ and $q=r= 4$ to get
\[
\|\La^{- \f{\al}{2}} [\La^{\f{\al}{2}}, U_F . \nabla] F\|_{L^2} \leq \|F\|^2_{L^4}
\]
then
\begin{eqnarray*}
I_1^f &\leq& \epsilon_0^{\al} C \|\La^{\al} F\|_{L^2} \|F\|^2_{L^4} \leq \f{1}{100} \|F\|^4_{L^4} + C \f{\epsilon_0^{2 \al}}{100} \|\La^{\al} F\|^2_{L^2}\\
&\leq& \f{1}{100} \|F\|^4_{L^4} +  \f{\epsilon_0^{ \al- \be}}{100} \|\La^{\al} F\|^2_{L^2}
\end{eqnarray*}
where we took 
\[
\epsilon_0 \leq \f{1}{100 C}
\]
{\bf Estimate for $I_1^{\theta}$:} 
\[
I_1^{\theta} = \epsilon_0^{\al} |\langle [\La^{\f{\al}{2}}, U_{\Theta} . \nabla] F, \La^{\f{\al}{2}} F\rangle| \leq \epsilon_0^{\al} \|\La^{\al} F\|_{L^2} \|\La^{- \f{\al}{2}} [\La^{\f{\al}{2}}, U_{\Theta} . \nabla] F\|_{L^2}
\]
if in \ref{25} we take $s_1= s_2= \f{\al}{2}$, $s_3= \al$, $V= \Theta $, $\vp= F$, $a= 1$ then
\begin{eqnarray*}
\|\La^{- \f{\al}{2}} [\La^{\f{\al}{2}}, U_{\Theta} . \nabla] F\|_{L^2} &\leq& \|\Theta\|_{L^{\infty}} \|\La^{\be} F\|_{L^2}\\
&\leq& \|\Theta\|_{L^{\infty}} \|\La^{\al} F\|^{\f{\be}{\al}}_{L^2} \|F\|^{\f{\al- \be}{\al}}_{L^2}
\end{eqnarray*}
therefore
\[
I_1^{\theta} \leq \epsilon_0^{\al} C \|\La^{\al} F\|^{1+ \f{\be}{\al}}_{L^2}  \leq \f{\epsilon_o^{\al- \be}}{100} \|\La^{\al} F\|^2_{L^2}+ C_{\epsilon_0}
\]
\subsubsection{Estimate for $I_2$}~\\ 
\textup{Case $3- 4\alpha < 0$}: \\
In this case we have $3 (1- \alpha) < \alpha$, therefore 
\[
 |\langle \Lambda^{2(\beta- \alpha)+ \frac{\alpha}{2}} \partial_1 \Theta, \Lambda^{\frac{\alpha}{2}} F\rangle| \leq \|\Lambda^{3( 1- \alpha)} F\|_{L^2} \|\Theta\|_{L^2}
\]
and
\[
\|\Lambda^{3( 1- \alpha)} F\|_{L^2} \leq \|\Lambda^{ \alpha} F\|^a_{L^2} \|F\|^{1- a}_{L^2}
\]
where 
\[
a= \frac{3 (1- \alpha)}{\alpha} < 1
\]
therefore 
\[
I_2 \leq C \epsilon_0^{2- 3 \alpha} \|\Lambda^{ \alpha} F\|^a_{L^2} \leq \frac{\epsilon_0^{\al- \be}}{100} \|\Lambda^{ \alpha} F\|^2_{L^2}+ C_{\epsilon_0}
\]
\textup{Case $3- 4\alpha > 0$}:\\
 By H\"older
 \[
 |\langle \Lambda^{2(\beta- \alpha)+ \frac{\alpha}{2}} \partial_1 \Theta, \Lambda^{\frac{\alpha}{2}} F\rangle| \leq \|\Lambda^{\alpha} F\|_{L^2} \|\Lambda^{3- 4 \alpha} \Theta\|_{L^2}
 \]
 and then 
 \[
 \|\Lambda^{3- 4 \alpha} \Theta\|_{L^2} \leq \|\Lambda^{2\beta} \Theta\|^a_{L^2} \|\Theta\|^{1-a}_{L^{\infty}},
 \]
 where 
 \[
 a= \frac{3- 4 \alpha}{2 \beta}
 \]
 Therefore
 \[
 I_2 \leq \frac{\epsilon_0^{\al- \be}}{100} \|\Lambda^{\alpha} F\|^2_{L^2}+ C_{\epsilon_0} 
 \|\Lambda^{2 \beta} \Theta\|^{2 a}_{L^2}.
 \]
  Since $a <1$,
 \[
I_2 \leq \frac{\epsilon_0^{\al- \be}}{100} \|\Lambda^{\alpha} F\|^2_{L^2}+ \frac{1}{100} \|\Lambda^{2 \beta} \Theta\|^2_{L^2}+ C_{\epsilon_0}. 
 \]
 Considering the two sub-cases above, the last inequality is the proper estimate for $I_2$.
 \\
\subsubsection{Estimate for $I_3$}~\\
{\bf Estimate for $I_3^{\theta}$:}
 \begin{eqnarray*}
 I_3^{\theta}= \epsilon_0  |\dpr{\Lambda^{\frac{\alpha}{2}} [R_{\alpha}, U_{\Theta} . \nabla] \Theta}{\Lambda^{\frac{\alpha}{2}} F}| &\leq& \epsilon_0 \|\Lambda^{\alpha} F\|_{L^2} \| [R_{\alpha}, U_{\Theta} . \nabla] \Theta\|_{L^2} 
 \end{eqnarray*}
   Now if in \ref{20} we take $s_1= 0$, $s_2= \beta$, $V= \La^{-\al} \Theta$, $a=1$, $p= 2$ and $q= r= 4$ then
  \[
  \| [R_{\alpha}, U_{\Theta} . \nabla] \Theta\|_{L^2} \leq \|\La^{\beta} \Theta\|^2_{L^4} \leq (\|\La^{2 \beta} \Theta\|^{\f{1}{2}}_{L^2} \|\Theta\|^{\f{1}{2}}_{L^{\infty}})^2
  \]
  therefore
  \[
  I_3^{\theta} \leq \epsilon_0 \|\Theta\|^{\f{3}{2}}_{L^{\infty}} \|\La^{\al} F\|_{L^2} \|\La^{2 \beta} \Theta\|_{L^2} \leq \f{\epsilon_0^{\al- \be}}{100} \|\La^{\al} F\|^2_{L^2}+ \f{1}{100} \|\La^{2 \beta} \Theta\|^2_{L^2}
  \]
  where we took
 \[
 \epsilon_0 \leq (\frac{1}{10^4  \|\theta_0\|_{L^{\infty}}})^{\frac{1}{3-2\al}}.
 \]
\\
{\bf Estimate for $I_3^f$:}
 \begin{eqnarray*}
 I_3^f= \epsilon_0  |\dpr{\Lambda^{\frac{\alpha}{2}} [R_{\alpha}, U_F. \nabla] \Theta}{\Lambda^{\frac{\alpha}{2}} F}| &\leq& \epsilon_0 \|\Lambda^{\alpha} F\|_{L^2} \| [R_{\alpha}, U_F. \nabla] \Theta\|_{L^2} \\
 &\leq & \f{\epsilon_0^{\al- \be}}{100} \|\Lambda^{\alpha} F\|^2_{L^2}+ C \epsilon_0^{3- 2 \al} \| [R_{\alpha}, U_F. \nabla] \Theta\|^2_{L^2}
 \end{eqnarray*}
Now by applying \ref{20}  with  $s_1=0, p=2, q=r=4$, $s_2=1-\al$, $a=1$, we have
 \[
 \epsilon_0^{3-2\al} \| [R_{\alpha}, U_F. \nabla] \Theta\|^2_{L^2}\leq  \epsilon_0^{3-2\al} \|F \|^2_{L^4} \|\Lambda^{1- \alpha} \Theta\|^2_{L^4} \leq \frac{1}{100}\|F \|^4_{L^4}+ 
 C \epsilon_0^{6-4\al}
 \|\Lambda^{\beta} \Theta\|^4_{L^4}
 \] 
 and
 \[
 \|\Lambda^{\beta} \Theta\|_{L^4} \leq \|\Lambda^{2 \beta} \Theta\|^{\frac{1}{2}}_{L^2} \|\Theta\|^{\frac{1}{2}}_{L^{\infty}}.
 \]
 If we take $\epsilon_0$ so that
 \[
 \epsilon_0 \leq (\frac{1}{100 C \|\theta_0\|_{L^{\infty}}} )^{\frac{1}{6-4\al}}
 \]
 then
 \begin{eqnarray*}
 \epsilon_0^{3-2\al} \| [R_{\alpha}, U_f. \nabla] \Theta\|^2_{L^2} &\leq & \ \frac{1}{100} \|F\|^4_{L^4}+ \frac{1}{100} \|\Lambda^{2 \beta} \Theta\|^2_{L^2},
  \end{eqnarray*}
   therefore
 \[
 I_3^f \leq \f{\epsilon_0^{\al- \be}}{100} \|\Lambda^{\alpha} F\|^2_{L^2}+ \frac{1}{100} \|F\|^4_{L^4}+ \frac{1}{100} \|\Lambda^{2 \beta} \Theta\|^2_{L^2}. 
 \]
 \subsubsection{Estimate for $I_5$}~\\
{\bf Estimate for $I_5^{\theta}$:}
 Apply \ref{25} with $s_1=\be/2, s_2=3\be/2$ and $s_3=\al$, 
 \begin{equation}
 \label{g:100}
 I_5^{\theta}= \epsilon_0^{\al} |\langle  [\Lambda^{\frac{3 \beta}{2}}, U_\Theta \cdot \nabla] \Theta, \Lambda^{\frac{3 \beta}{2}} \Theta \rangle| \leq \epsilon_0^{\al} C\|\theta_0\|_{L^\infty}  \|\La^{2\be} \Theta\|_{L^2}^2 \leq \f{1}{100} \|\La^{2\be} \Theta\|_{L^2}^2
 \end{equation}
 where we took
 \[
 \epsilon_0 \leq (\f{1}{100 C \|\Theta_0\|_{L^{\infty}}})^{\f{1}{\al}}
 \]
{\bf Estimate for $I_5^f$:}
 \begin{eqnarray*}
 I_5^f= \epsilon_0^{\alpha} |\langle  [\Lambda^{\frac{3 \beta}{2}}, U_F \cdot  \nabla] \Theta, \Lambda^{\frac{3 \beta}{2}} \Theta \rangle| &=& \epsilon_0^{\alpha} |\langle \Lambda^{\frac{-\beta}{2}} [\Lambda^{\frac{3 \beta}{2}}, U_F \cdot  \nabla] \Theta, \Lambda^{2 \beta} \Theta \rangle| \\
 &\leq & \frac{1}{100} \|\Lambda^{2 \beta} \Theta\|^2_{L^2}+ C \epsilon_0^{2 \alpha} \|\Lambda^{\frac{-\beta}{2}} [\Lambda^{\frac{3 \beta}{2}}, U_F\cdot \nabla] \Theta\|^2_{L^2}
 \end{eqnarray*}
 We apply \ref{20} with $s_1= \frac{\beta}{2}$, $s_2= \frac{3 \beta}{2}$, $V= U_F$, $\phi= \Theta$ and $q= r= 4$ then
 \begin{eqnarray*}
 \epsilon_0^{2 \al} C \|\Lambda^{\frac{-\beta}{2}} [\Lambda^{\frac{3 \beta}{2}}, U_F. \nabla] \Theta\|^2_{L^2} \leq \epsilon_0^{2 \al} C \|F\|^2_{L^4} \|\Lambda^{\beta} \Theta \|^2_{L^4} \leq \frac{1}{100} \|F\|^4_{L^4}+ \epsilon_0^{4 \al} CC \|\Lambda^{\beta} \Theta \|^4_{L^4}
 \end{eqnarray*}
 and
 \[
  \|\Lambda^{\beta} \Theta \|_{L^4} \leq  \|\Lambda^{2 \beta} \Theta \|^{\frac{1}{2}}_{L^2} \|\Theta \|^{\frac{1}{2}}_{L^{\infty}},
 \]
 Therefore 
 \[
 \epsilon_0^{2 \alpha} \|\Lambda^{-\frac{\beta}{2}} [\Lambda^{\frac{3 \beta}{2}}, U_F . \nabla] \Theta\|^2_{L^2} \leq \frac{1}{100} \|F\|^4_{L^4}+ C \epsilon_0^{2 \alpha} \| \Theta_0\|^2_{L^{\infty}} \|\Lambda^{2 \beta} \Theta \|^2_{L^2}.
 \]
 From here we take $\epsilon_0$ so small that 
 \[
 \epsilon_0 \leq (\frac{1}{100 C \|\Theta_0\|^2_{L^{\infty}}})^{\frac{1}{4 \alpha}}.
 \]
 to get
 \[
 I_5^f \leq \frac{1}{100} \|F\|^4_{L^4}+ \frac{1}{50} \|\Lambda^{2 \beta}\|^2_{L^2}
 \] 
Now putting all the above estimates together along with a using of Gronwall's inequality finishes the proof for $L^4$.
\end{proof}
\subsection{$L^6$ Estimate}~\\
Now we have enough information of $\theta$ and $f$ to get the $L^6$ estimate

\begin{proposition}
\label{propo:4}
Let $\alpha > \frac{2}{3}$, then for any $T > 0$ there exists a $C_T$ such that
\[
\sup_{0 \leq t \leq T} \|F\|^6_{L^6}+ \int_0^T \|F\|^6_{L^{\frac{12}{2- \alpha}}} dt \leq C_T,
\]

\[
\sup_{0 \leq t \leq T} \|\La^{\f{1+ \beta}{2}}F\|^2_{L^2}+ \int_0^T \| \partial F\|^2_{L^2} dt \leq C_T,
\]
and  
\[
\sup_{0 \leq t \leq T} \|\La^{\f{5 \beta}{2}} \Theta\|^2_{L^2}+ \int_0^T \| \La^{3 \beta} \Theta\|^2_{L^2} dt \leq C_T.
\]
\end{proposition}

\begin{proof}
In \ref{21}, \ref{641} and \ref{65} take $p= 6$, $s= \f{1+ \be}{2}$ and $\kappa= \f{5 \be}{2}$ to get

\begin{eqnarray*}
\partial_t (\f{1}{6} \| F \|^6_{L^6}+ \f{1}{2} \|\La^{\f{1+ \be}{2}} F \|^2_{L^2}+ \f{1}{2} \|\La^{\f{ 5 \be}{2}} \Theta \|^2_{L^2})
&+& \epsilon_0^{\al- \be} \| F \|^6_{L^{\f{12}{2- \al}}}+ \epsilon_0^{\al- \be} \|\partial F \|^2_{L^2}+  \|\La^{3 \be} \Theta \|^2_{L^2} \\
&\leq& K_1+ K_2+ K_3+ I_1+ I_2+ I_3+ I_4+ I_5
\end{eqnarray*}
\subsubsection{Estimate for $K_1$}~\\
We consider two cases:\\
{\bf Case 1: $3-4\alpha \geq 0$}:\\
By Holder inequality
\[
K_1 \leq \epsilon_0^{2- 3 \alpha} |\int F^5 \Lambda^{3- 4 \alpha} \Theta dx| \leq \epsilon_0^{2- 3 \alpha} \|F\|^5_{L^{\frac{12}{2- \alpha}}} \|\Lambda^{3- 4 \alpha} \Theta \|_{L^{\frac{12}{5 \alpha+ 2}}}.
\]
By Sobolev inequality
\[
\|\Lambda^{3- 4 \alpha} \Theta \|_{L^{\frac{12}{5 \alpha+ 2}}} \leq \|\Lambda^{\frac{22-29 \alpha}{6}} \Theta \|_{L^2}.
\]
Now since for $\f{2}{3} < \al \leq \f{3}{4}$, $0 < \frac{22-29 \alpha}{6} < \f{3 \be}{2}$, there is a $0< \gamma < 1$ such that
\[
\|\Lambda^{\frac{22-29 \alpha}{6}} \Theta \|_{L^2} \leq \|\La^{\f{3\be}{2}} \Theta \|^{\gamma}_{L^2} \|\Theta \|^{1- \gamma}_{L^2}= C,
\]
therefore
\[
K_1 \leq \epsilon_0^{2- 3 \alpha} C \|F\|^5_{L^{\frac{12}{2- \alpha}}} \leq \frac{\epsilon_0^{\alpha- \beta}}{100} \|F\|^6_{L^{\frac{12}{2- \alpha}}} + C_{\epsilon_0}
\]
{\bf Case 2: $3-4\alpha < 0$}:\\
Use  Holder and Sobolev inequalities to get
\begin{eqnarray*}
K_1 &\leq& \epsilon_0^{2- 3\alpha} \|F\|^5_{L^6} \|\La^{3- 4\alpha} \Theta\|_{L^6} \leq \epsilon_0^{2- 3\alpha} \|F\|^5_{L^6} \|\Theta\|_{L^{\frac{6}{12 \alpha- 8}}}\\
&\leq& \f{1}{100} \|F\|^6_{L^6} + C_{\epsilon_0}
\end{eqnarray*}
Now we put both cases together to get
\[
K_1 \leq \frac{\epsilon_0^{\alpha- \beta}}{100} \|F\|^6_{L^{\frac{12}{2- \alpha}}}+ \frac{1}{100} \|F\|^6_{L^6}+ C_{\epsilon_0}
\]
\subsubsection{Estimate for $K_2$}~\\
{\bf Estimate for $K^f_2$:} \\
In \ref{20} take $s_1= 0$, $s_2= \beta$, $a= 1$, $V= \La^{-1}F$, $\phi= \Theta$, $q= \frac{12}{2-\alpha}$ and $r= \f{2}{\al}$ to get
\begin{eqnarray*}
K^f_2 = \epsilon_0 |\langle [R_{\alpha}, U_F . \nabla] \Theta, F^5\rangle| &\leq& \epsilon_0 \|F\|^5_{L^{\f{12}{2- \al}}} \|[R_{\al}, U_F . \nabla] \Theta\|_{L^{\f{12}{5 \al+ 2}}} \\
&\leq& \epsilon_0 \|F\|^6_{L^{\frac{12}{2- \alpha}}} \| \La^{\beta} \Theta \|_{L^{\frac{2}{\alpha}}}
\end{eqnarray*}
then
\[
\|\La^{\be} \Theta\|_{L^{\f{2}{\al}}} \leq \|\La^{\f{3 \be}{2}} \Theta \|^{\f{2}{3}}_{L^2} \|\Theta\|^{\f{1}{3}}_{L^{\f{2}{3 \al- 2}}}= C
\]
therefore
\[
K_2^f \leq \epsilon_0 C \|F \|^6_{L^{\f{12}{2- \al}}} \leq \f{\epsilon_0^{\al- \be}}{100} \|F \|^6_{L^{\f{12}{2- \al}}}
\]
where we took $\epsilon_0$ such that
\[
\epsilon_0 \leq [\frac{1}{100  C}]^{\frac{1}{2 \beta}}.
\]
{\bf Estimate for $K^{\theta}_2$:} 
\[
K^{\theta}_2= \epsilon_0 |\langle [R_{\alpha}, U_{\Theta} . \nabla] \Theta, F^5\rangle| \leq \epsilon_0 \|\La^{\be} (F^5)\|_{L^{\f{12}{7 (2- \al)}}} \|\La^{- \be}[R_{\al},U_{\Theta} . \nabla] \Theta\|_{L^{\f{12}{7 \al- 2}}},
\]
by Kato-Ponce, 
\[
\|\La^{\be} (F^5)\|_{L^{\f{12}{7 (2- \al)}}} \leq C \|\La^{\be} F\|_{L^{\f{4}{ 2- \al}}} \| F\|^4_{L^{\f{12}{2- \al}}}
\]
then
\[
\|\La^{\be} F\|_{L^{\f{4}{ 2- \al}}} \leq \|\partial F\|^{\be}_{L^2} \|F\|^{\al}_{L^4},
\]
and if \ref{20} we take $s_1= s_2= \be$, $V= \La^{- \al} \Theta$, $\phi= \Theta$, $a= \al+ \f{\be}{2}$ and $q= r= \f{24}{7 \al- 2}$ we have
\[
\|\La^{- \be}[R_{\al},U_{\Theta} . \nabla] \Theta\|_{L^{\f{12}{7 \al- 2}}}  \leq \|\La^{\f{\be}{2}} \Theta\|^2_{L^{\f{24}{7 \al- 2}}} \leq (\|\La^{3 \be} \Theta\|^{\f{1}{6}}_{L^2} \|\Theta \|^{\f{5}{6}}_{L^{\f{20}{7 \al- 4}}})^2
\]
therefore
\begin{eqnarray*}
K_2^{\theta} &\leq& \epsilon_0 C \|F \|^4_{L^{\f{12}{2- \al}}} \|\partial F\|^{\be}_{L^2} \|\La^{3 \be} \Theta\|^{\f{1}{3}}_{L^2}\\
&\leq & \f{\epsilon^{\al- \beta}_0}{100}
 \|F\|^6_{L^{\f{12}{2- \al}}}+ (\epsilon^{\f{1+ 4 \be}{3}}_0 C  
 \|\partial F\|^{\be}_{L^2} \|\La^{3 \be} \Theta\|^{\f{1}{3}}_{L^2})^3,
\end{eqnarray*}
now since $3 (\be + \f{1}{3}) < 2$, 
\[
K_2^{\theta} \leq  \f{\epsilon^{\al- \beta}_0}{100} \|F\|^6_{L^{\f{12}{2- \al}}}+ \f{\epsilon^{\al- \be}_0}{100} \|\partial F\|^2_{L^2} + \f{1}{100} \|\La^{3 \be} \Theta\|^2_{L^2}+ C_{\epsilon_0}
\]
\subsubsection{Estimate for $I_1$}~\\
{\bf Estimate for $I^f_1$:} \\
\begin{eqnarray*}
I^f_1 &=& \epsilon_0^{\al} | \langle [\La^{\f{1+ \be}{2}}, U_F . \nabla] F, \La^{\f{1+ \be}{2}} F \rangle| \leq \epsilon_0^{\al} \|\partial F\|_{L^2} \| \La^{- \f{\al}{2}} [\La^{\f{1+ \be}{2}}, U_F . \nabla] F\|_{L^2}\\
&\leq& \f{\epsilon_0^{\al- \be}}{100} \|\partial F\|^2_{L^2}+ \epsilon_0 C \| \La^{- \f{\al}{2}} [\La^{\f{1+ \be}{2}}, U_F . \nabla] F\|^2_{L^2},
\end{eqnarray*}
then if in \ref{20} we take $s_1= \f{\al}{2}$, $s_2= \f{1+ \be}{2}$, $V= \La^{-1} F$, $\phi= F$, $a= 1$ and $q= r= 4$, then a using of \ref{20}, Sobolev inequality and Gagliardo-Nirenberg gives
\begin{eqnarray*}
\| \La^{- \f{\al}{2}} [\La^{\f{1+ \be}{2}}, U_F . \nabla] F\|_{L^2} &\leq& \|F\|_{L^4} \|\La^{\be} F\|_{L^4}= C \|\La^{\be} F\|_{L^4} \\
&\leq & C \|\La^{\f{1+ 2 \be}{2}} F\|_{L^2} \leq C \|\partial F\|^{\f{1+ 2 \be}{2}}_{L^2} \|F\|^{\f{1- 2 \be}{2}}_{L^2}
\end{eqnarray*}
therefore
\begin{eqnarray*}
I^f_1 &\leq& \f{\epsilon_0^{\al- \be}}{100} \|\partial F\|^2_{L^2}+ \epsilon_0 C  \|\partial F \|^{1+ \be}_{L^2} \\
&\leq& \f{\epsilon_0^{\al- \be}}{50} \|\partial F\|^2_{L^2}+ C_{\epsilon_0}.
\end{eqnarray*}
{\bf Estimate for $I^{\theta}_1 $:} \\
\begin{eqnarray*}
I^{\theta}_1 &=& \epsilon_0^{\al} | \langle [\La^{\f{1+ \be}{2}}, U_{\Theta} . \nabla] F, \La^{\f{1+ \be}{2}} F \rangle| \leq \epsilon_0^{\al} \|\partial F\|_{L^2} \| \La^{- \f{\al}{2}} [\La^{\f{1+ \be}{2}}, U_{\Theta} . \nabla] F\|_{L^2}\\
&\leq& \f{\epsilon_0^{\al- \be}}{100} \|\partial F\|^2_{L^2}+ \epsilon_0 C \| \La^{- \f{\al}{2}} [\La^{\f{1+ \be}{2}}, U_{\Theta} . \nabla] F\|^2_{L^2},
\end{eqnarray*}
now if in \ref{25} we take $s_1= \f{\al}{2}$, $s_2= \f{1+ \be}{2}$, $s_3= \al$, $V= \Theta$, $\phi= F$ then Gagliardo-Nirenberg yields
\begin{eqnarray*}
\| \La^{- \f{\al}{2}} [\La^{\f{1+ \be}{2}}, U_{\Theta} . \nabla] F\|_{L^2} &\leq& \|\Theta\|_{L^{\infty}} \|\La^{2 \be} F\|_{L^2}= C \|\La^{2 \be} F\|_{L^2} \\
&\leq & C \|\partial F\|^{2 \be}_{L^2} \|F\|^{1- 2 \be}_{L^2}
\end{eqnarray*}
therefore
\begin{eqnarray*}
I^{\theta}_1 &\leq& \f{\epsilon_0^{\al- \be}}{50} \|\partial F\|^2_{L^2}+ C_{\epsilon_0}.
\end{eqnarray*}
\subsubsection{Estimate for $I_2$}~\\
\[
I_2= \epsilon_0^{2- 3 \al} |\langle \La^{2 (\be- \al)+ \f{1+ \be}{2}} \partial_1 \Theta, \La^{\f{1+ \be}{2}} F\rangle| \leq \epsilon_0^{2- 3 \al} \|\partial F\|_{L^2} \|\La^{4- 5 \al} \Theta\|_{L^2}.
\]
To find the bound for the right hand side,we consider two cases.\\
{\bf Case $1$: $4- 5 \al \geq 0$},\\
In this case since, $4- 5 \al \leq 3 \be$, there is a $0 < \gamma< 1$, such that
\[
\|\La^{4- 5 \al} \Theta\|_{L^2} \leq \|\La^{3 \be} \Theta\|^{\gamma}_{L^2} \|\Theta\|^{1- \gamma}_{L^2},
\]
therefore in this case
\[
I_2 \leq \epsilon_0^{2- 3 \al} C \|\partial F\|_{L^2} \|\La^{3 \be} \Theta\|^{\gamma}_{L^2} \leq \f{\epsilon_0^{\al- \be}}{100} \|\partial F\|^2_{L^2}+ \f{1}{100} \|\La^{3 \be} \Theta\|^2_{L^2}+ C_{\epsilon_0}
\]
{\bf Case $2$: $4- 5 \al < 0$},\\
In this case by Sobolev inequality we have 
\[
\|\La^{4- 5 \al} \Theta\|_{L^2} \leq \|\Theta\|_{L^{\f{2}{5 \al- 3}}}.
\]
Note that $\f{2}{5 \al- 3}  \geq 1$, hence
\[
I_2 \leq \epsilon_0^{2- 3 \al} C \|\partial F\|_{L^2} \leq \f{\epsilon_0^{\al- \be}}{100} \|\partial F\|^2_{L^2}+ C_{\epsilon_0}.
\]
\subsubsection{Estimate for $I_3$}~\\
{\bf Estimate for $I^f_3$:} \\
\begin{eqnarray*}
I_3^f &=& \epsilon_0 |\langle \La^{\f{1+ \be}{2}} [R_{\al}, U_F . \nabla] \Theta, \La^{\f{1+ \be}{2}} F\rangle| \leq \epsilon_0 \|\partial F\|_{L^2} \|\La^{\be} [R_{\al}, U_F . \nabla] \Theta\|_{L^2}\\
&\leq& \f{\epsilon_0^{\al- \be}}{100} \|\partial F\|^2_{L^2} + \epsilon^{1+ 2 \be}_0 C \|\La^{\be} [R_{\al}, U_F . \nabla] \Theta\|^2_{L^2}
\end{eqnarray*}
Now in \ref{200} take $s= \be$, $V= U_F$, $\vp= \Theta$, $a= 1$, $q= 6$, and $r= 3$ to get
\[
\|\La^{\be} [R_{\al}, U_F . \nabla] \Theta\|_{L^2} \leq  \|F\|_{L^6} \|\La^{2 \be} \Theta\|_{L^3} 
\]
then
\begin{eqnarray*}
\epsilon^{1+ 2 \be}_0 C \|\La^{\be} [R_{\al}, U_F . \nabla] \Theta\|^2_{L^2} 
&\leq& \f{1}{100} \|F\|^6_{L^6}+ \epsilon_0^{\f{3 (1+ 2 \be)}{2}} \|\La^{2 \be} \Theta\|^3_{L^3}.
\end{eqnarray*}
Now 
\[
\|\La^{2 \be} \Theta\|_{L^3} \leq \|\La^{3 \be} \Theta\|^{\f{2}{3}}_{L^2} \| \Theta\|^{\f{1}{3}}_{L^{\infty}}
\]
then we take 
\[
\epsilon_0 \leq (\f{1}{100 C \|\Theta_0\|_{L^{\infty}}})^{\f{2}{3 (1+ 2 \be)}}
\]
to get
\[
I_3^f \leq  \f{1}{100} \|F\|^6_{L^6}+ \f{\epsilon_0^{\al- \be}}{100} \|\partial F\|^2_{L^2}+ \f{1}{100} \|\La^{3 \be} \Theta\|^2_{L^2}. 
\]
{\bf Estimate for $I^{\theta}_3$:} \\
\begin{eqnarray*}
I_3^{\theta} &=& \epsilon_0 |\langle \La^{\f{1+ \be}{2}} [R_{\al}, U_{\Theta} . \nabla] \Theta, \La^{\f{1+ \be}{2}} F\rangle| \leq \epsilon_0 \|\partial F\|_{L^2} \|\La^{\be} [R_{\al}, U_{\Theta} . \nabla] \Theta\|_{L^2}\\
&\leq& \f{\epsilon_0^{\al- \be}}{100} \|\partial F\|^2_{L^2} + \epsilon^{1+ 2 \be}_0 C \|\La^{\be} [R_{\al}, U_{\Theta} . \nabla] \Theta\|^2_{L^2}.
\end{eqnarray*}
Now in \ref{200} take $s= \be$, $V= U_{\Theta}$, $\vp= \Theta$, $a= 1$, $q= 6$, and $r= 3$ to get
\begin{eqnarray*}
\|\La^{\be} [R_{\al}, U_{\Theta} . \nabla] \Theta\|_{L^2} &\leq& \|\La^{\be} \Theta\|_{L^6} \|\La^{2 \be} \Theta\|_{L^3}\\
&\leq& ( \|\La^{3 \be} \Theta\|^{\f{1}{3}}_{L^2} \|\Theta\|^{\f{2}{3}}_{L^{\infty}}) ~ (  \|\La^{3 \be} \Theta\|^{\f{2}{3}}_{L^2} \|\Theta\|^{\f{1}{3}}_{L^{\infty}})\\
&=& \|\La^{3 \be} \Theta\|_{L^2} \|\Theta\|_{L^{\infty}}
\end{eqnarray*}
therefore if we choose
\[
\epsilon_0 \leq (\f{1}{100 C \|\Theta_0\|^2_{L^{\infty}}})^{\f{1}{1+ 2 \be}}
\]
we get
\[
I_3^{\theta} \leq \f{\epsilon_0^{\al- \be}}{100} \|\partial F\|^2_{L^2}+ \f{1}{100} \|\La^{3 \be} \Theta\|^2_{L^2}
\]
{\bf Estimate for $I_5$}~\\
{\bf Estimate for $I^f_5$:} \\
\begin{eqnarray*}
I_5^f &=& \epsilon_0^{\al} |\langle [\La^{\f{5 \be}{2}}, U_F . \nabla] \Theta, \La^{\f{5\be}{2}} \Theta\rangle| \leq \epsilon_0^{\al} \|\La^{3 \beta} \Theta\|_{L^2} \|\La^{- \f{\be}{2}} [\La^{\f{5 \be}{2}}, U_F . \nabla] \Theta \|_{L^2}\\
&\leq& \f{1}{100} \|\La^{3 \be} \Theta\|^2_{L^2}+ \epsilon_0^{2 \al} C \|\La^{- \f{\be}{2}} [\La^{\f{5 \be}{2}}, U_F . \nabla] \Theta \|^2_{L^2}\\
&\leq& \f{1}{100} \|\La^{3 \be} \Theta\|^2_{L^2}+ \epsilon_0^{2 \al} C \|F\|^2_{L^6} \|\La^{2 \be} \Theta \|^2_{L^3}\\
&\leq& \f{1}{100} \|\La^{3 \be} \Theta\|^2_{L^2}+ \f{\epsilon_0^{\al- \be}}{100}  \|F\|^6_{L^6} + \epsilon_0^{\f{1+ 4 \al}{2}} C \|\La^{2 \be} \Theta \|^3_{L^2}\\
&\leq&  \f{1}{100} \|\La^{3 \be} \Theta\|^2_{L^2}+ \f{\epsilon_0^{\al- \be}}{100}  \|F\|^6_{L^6} + \epsilon_0^{\f{1+ 4 \al}{2}} C \|\La^{3 \be} \Theta \|^2_{L^2} \|\Theta\|_{L^{\infty}}\\
&\leq&  \f{1}{100} \|\La^{3 \be} \Theta\|^2_{L^2}+ \f{\epsilon_0^{\al- \be}}{100}  \|F\|^6_{L^6} + \f{1}{100} \|\La^{3 \be} \Theta \|^2_{L^2}
\end{eqnarray*}
where we took 
\[
\epsilon_0 \leq (\f{1}{100 C \|\Theta_0\|_{L^{\infty}}})^{\f{2}{1+ 4\al}}.
\]
{\bf Estimate for $I^{\theta}_5$:} \\
 A using of \ref{25} and Gagliardo-Nirenberg gives
\begin{eqnarray*}
I_5^f &=& \epsilon_0^{\al} |\langle [\La^{\f{5 \be}{2}}, U_{\Theta} . \nabla] \Theta, \La^{\f{\be}{2}} \Theta\rangle| \leq \epsilon_0^{\al} \|\La^{3 \beta} \Theta\|_{L^2} \|\La^{- \f{\be}{2}} [\La^{\f{5 \be}{2}}, U_{\Theta} . \nabla] \Theta \|_{L^2}\\
&\leq& \epsilon_0^{\al} \|\La^{3 \beta} \Theta\|^2_{L^2} \|\Theta\|_{L^{\infty}} \\
&\leq& \f{1}{100} \|\La^{3 \beta} \Theta\|^2_{L^2}
\end{eqnarray*}
where we took
\[
\epsilon_0 \leq (\f{1}{100 C \|\Theta_0\|_{L^{\infty}}})^{\f{1}{\al}}
\]
which completes the proof.
\end{proof}

\appendix

\section{Commutator estimates}
Before we proceed with the proofs of Lemma \ref{le:com1} and Lemma  \ref{le:com2}, we would like to present some classical estimates for maximal functions, 
which will be used frequently in this section.  First, there  is the point-wise control of Littlewood-Paley operators by the maximal function, namely 
 $$
 (\De_k f) (x)+ (\De_{<k-10} f)(x) \leq C \cm[f](x).
 $$
Another useful result is the  Fefferman-Stein estimate for the maximal function  (see Theorem 4.6.6, p. 331, \cite{Graf}), which states that $\cm$ is a bounded operator from $L^p(l^r)$ into itself. More explicitly, for every $r,p\in (1, \infty)$, there is $C_{p,r}$, so that 
$$
\| (\sum_k (\cm g_k )^r)^{1/r}\|_{L^p}\leq C_{p,r} \| (\sum_k |g_k|^r)^{1/r}\|_{L^p}
$$
Another basic tool is   the following standard para product decomposition 
$$
\Delta_k(f g)=   \Delta_k( f_{<k-10}  g_{\sim k})+ \Delta_k(f_{\sim k} g_{<k+10})+\Delta_k(\sum_{l=k+10}^\infty f_l g_{\sim l}),
$$
available for say every pair of Schwartz functions $f,g$. 
We refer to the corresponding terms as low-high, high-low and high-high interaction terms. 

In what follows, we present the proof of \ref{20} and \ref{201}. The difference between the two estimates is only in the dependence on the derivatives 
 $\pm s_1$ taken on the commutators. Below, we take $\La^{-s_1}$ (matching the setup in \ref{20}), but  we assume $s_1\in (-1,1)$ as to cover both \ref{20} and \ref{201}. A crucial condition that needs to be met though is that $s_2-s_1\leq 1$. 
\subsection{Proof of \ref{20} and \ref{201}} 
We first present the proof for the hardest  case $a=1$. We then discuss the necessary adjustments for the general case $a\in [s_2-s_1,1)$. 
Start with 
$$
 \La^{-s_1}[\La^{s_2}, V \cdot \nabla] \vp]=\sum_k \De_k[\La^{-s_1}[\La^{s_2}, V \cdot \nabla] \vp]].
$$

 Each one of these terms generates a separate entry for the estimate \ref{20}.  
 
 \subsubsection{Low-high terms}  For the low-high term, which is usually the hardest one in commutator estimates theory, we need to estimate $\|I_{low,high}\|_{L^p}$, where 
 $$
 I_{low, high}(x)= \sum_k \De_k[\La^{-s_1}[\La^{s_2}, V_{<k-10} \cdot \nabla] \vp_{\sim k}]]  
  $$
  In fact, we will show the estimate only under the restriction 
 $2<q\leq \infty$ and {\bf no restrictions on $s_2, s_1$}. More precisely, $q=\infty$ and any $s_1, s_2$ are allowed for the low-high interaction terms. Below, we tacitly assume $q<\infty$, the proof for $q=\infty$ requires minor modifications, which are left to the reader. 
  By Littlewood-Paley  theory, it suffices to control $\|S\|_{L^p}$, where the  Littlewood-Paley square function $S$ is given by 
 \begin{eqnarray*}
S^2(x) &=&  \sum_k |\De_k[\La^{-s_1}[\La^{s_2}, V_{<k-10} \cdot \nabla] \vp_{\sim k}]](x)|^2= \\
&=&
 \sum_k 2^{2 k(s_2- s_1)} |\De^1_k[[\De_k^2, V_{<k-10} \cdot \nabla] \vp_{\sim k}]](x)|^2,
 \end{eqnarray*}
 where $\De_k^{j}, j=1,2$ are modified Littlewood-Paley operators similar to $\De_k$. We will show that for $p_1, q_1\in (1, \infty): \f{1}{p_1}+\f{1}{q_1}=1$, we have the pointwise bound 
 \begin{equation}
 \label{g:50}
| [\De_k^2, g \cdot \nabla] f ](x)|\leq C  \cm[|\nabla g|^{q_1}](x)^{1/q_1} \cm[|f|^{p_1}](x)^{1/p_1}.
 \end{equation}
 where $\nabla\cdot g=0$ and $\cm$ is the Hardy-Littlewood maximal function. 
 
 Assuming \ref{g:50}, let us show the estimate for the low-high piece of \ref{20}. We have for all $p_1, q_1\in (1, \infty): \f{1}{p_1}+\f{1}{q_1}=1$ 
 \begin{eqnarray*}
 S^2(x)  &\leq &  \sum_k 2^{2k(s_2-s_1)} |\De^1_k[[\De_k^2, V_{<k-10} \cdot \nabla] \vp_{\sim k}]](x)|^2\leq \\
 &\leq & \sum_k 2^{2k(s_2-s_1)} \cm[[\De_k^2, V_{<k-10} \cdot \nabla] \vp_{\sim k}]^2 \\
 &\leq & C \sum_k 2^{2k(s_2-s_1)} |\cm[\cm[|\nabla V_{<k-10}|^{q_1}]^{1/q_1} 
 \cm[|\vp_{\sim k}|^{p_1}]^{1/p_1}]^2.
 \end{eqnarray*}
 Clearly, $\cm[|\nabla V_{<k-10}|^{q_1}]\leq C \cm[[\cm(\nabla V)]^{q_1}]$. Thus, by the Fefferman-Stein estimates  and by the  H\"older's inequality 
 \begin{eqnarray*}
 \|S\|_{L^p} &\leq &  C  \|(\sum_k 2^{2k(s_2-s_1)} |\cm[\cm[|\nabla V_{<k-10}|^{q_1}]^{1/q_1} 
 \cm[|\vp_{\sim k}|^{p_1}]^{1/p_1}|^2)^{1/2}\|_{L^p}\leq \\
 &\leq & C  \| \cm[\cm(\nabla V)]^{q_1}])^{1/q_1} (\sum_k 2^{2k(s_2-s_1)} 
 \cm[|\vp_{\sim k}|^{p_1}]^{2/p_1})^{1/2}\|_{L^p}\leq  \\
 &\leq & \| \cm[[\cm(\nabla V)]^{q_1}]^{1/q_1}\|_{L^q} 
 \|(\sum_k 2^{2k(s_2-s_1)}\cm[|\vp_{\sim k}|^{p_1}]^{2/p_1}])^{1/2}\|_{L^r}.
 \end{eqnarray*}
 Here, we need to select $q_1<q$, so that we can estimate (by the boundedness of $\cm$ on $L^{q/q_1}$)
 \begin{eqnarray*} 
 \| \cm[[\cm(\nabla V)]^{q_1}]^{1/q_1}\|_{L^q} &=& \| \cm[[\cm(\nabla V)]^{q_1}]\|_{L^{q/q_1}}^{1/q_1}\leq 
 C \|\cm(\nabla V)^{q_1}\|_{L^{q/q_1}}^{1/q_1}\leq \\
 &\leq & C \| |\nabla V|^{q_1} \|_{L^{q/q_1}}^{1/q_1}=C\|\nabla V\|_{L^q}.
 \end{eqnarray*} 
 For the other term, let $p_1: p_1<2, p_1<r$. Upon introducing $g_k:=[2^{k(s_2-s_1)} |\vp_{\sim k}|]^{p_1}$, we have by Fefferman-Stein and Littlewood-Paley theory that 
 \begin{eqnarray*} 
& &   \|(\sum_k 2^{2k(s_2-s_1)}\cm[|\vp_{\sim k}|^{p_1}]^{2/p_1}])^{1/2}\|_{L^r}  =     \| (\sum_k |\cm g_k|^{2/p_1})^{1/2}\|_{L^{r}}\leq \\
  &\leq & 
   \| (\sum_k |\cm g_k|^{2/p_1})^{p_1/2}\|_{L^{r/p_1}}^{1/p_1} \leq C  \| (\sum_k |g_k|^{2/p_1})^{p_1/2}\|_{L^{r/p_1}}^{1/p_1} = \\
   &=& 
   C  \| (\sum_k 2^{2k(s_2-s_1)} |\vp_{\sim k}|^2)^{p_1/2}\|_{L^{r/p_1}}^{1/p_1} 
   = C  \| (\sum_k 2^{2k(s_2-s_1)} |\vp_{\sim k}|^2)^{1/2}\|_{L^{r}}\leq C \|\La^{s_2-s_1} \vp\|_{L^r}.
 \end{eqnarray*} 
 
   Analyzing the inequalities $p_1<2, p_1<r$ and $q_1<q$ shows that as long as $q>2$, we can always select $p_1, q_1: \f{1}{p_1}+\f{1}{q_1}=1$  with the required properties. This is easily seen by selecting $q_1=q-\epsilon, p_1=\f{q_1}{q_1-1}=\f{q-\epsilon}{q-1-\epsilon}$ for some small $\epsilon$. Thus, we have shown 
   \begin{equation}
   \label{36}
\|I_{low, high}\|_{L^p}\leq C \|\nabla V\|_{L^q} \|\La^{s_2-s_1} \vp\|_{L^r}.
 \end{equation}
   To finish the proof in this case, we need to prove \ref{g:50}. But this is a 
 simple application of the following representation formula for commutators
 \begin{eqnarray*}
 [\De_k^2,g\cdot \nabla] f(x) &=&  2^k [\De_k^3, g\cdot  ] f(x)=2^{3 k} \int_{\rtwo} 
 \chi_3(2^k(x-y)) [g(x)-g(y)]   f(y) dy=\\
 &=&  2^{3 k} \int_{\rtwo} 
 \chi_3(2^k(x-y)) (\int_0^1 \dpr{\nabla g(y+z(x-y))}{x-y} dz)   f(y) dy. 
 \end{eqnarray*}
 Clearly, after estimating this last expression, 
 \begin{eqnarray*}
 & & | [\De_k^2, V_{<k-10} \cdot \nabla] \vp_{\sim k}](x)| \leq   \\
 &\leq & C 2^{2k} \int_0^1 
 \int_{\rtwo} |\chi_4(2^k(x-y))| |\nabla g(y+z(x-y)| |f(y)| dy dz  \\
 &\leq & C \int_0^1 (\int_{\rtwo} 2^{2k} |\chi_4(2^k(x-y))| |f(y)|^{p_1} dy)^{1/p_1} \times \\
 &\times & 
 (\int_{\rtwo} 2^{2k} |\chi_4(2^k(x-y))| |\nabla g(y+z(x-y) |^{q_1} dy)^{1/q_1}, 
 \end{eqnarray*} 
 where $\chi_4(w)=\chi_3(w) w_i, i=1,2$. Clearly, 
 $$
 \int_{\rtwo} 2^{2k} |\chi_4(2^k(x-y))| |f(y)|^{p_1} dy\leq C \cm[|f|^{p_1}](x),
 $$
 Also, 
 \begin{eqnarray*}
 & & \int_{\rtwo} 2^{2k} |\chi_4(2^k(x-y))| |\nabla g(y+z(x-y) |^{q_1} dy=\int_{\rtwo} 2^{2k} 
 |\chi_4(2^k l)| |\nabla g(x-(1-z)l)|^{q_1} dl =  \\
 &=& \int_{\rtwo} \f{2^{2k}}{(1-z)^2} 
 |\chi_4(\f{2^k}{1-z} m)| |\nabla g(x-m)|^{q_1} dm\leq C \cm[|\nabla g|^{q_1}](x).
 \end{eqnarray*}
 This establishes \ref{g:50}.

 \subsubsection{High-low term} 
 Here, we need the assumption $s_2-s_1\leq 1$, but $q,r$ may be arbitrary (i.e. one does not $2<q$), as long as $\f{1}{p}=\f{1}{q}+\f{1}{r}$. 

In this case, the commutator structure does not play much role, so we just deal with the two terms separately. In fact, the term $\La^{-s_1}\De_k[V_{\sim k}\cdot \nabla \La^{s_2} \vp_{<k+10}]$ is simpler, so we omit its analysis. For the other term, we have by Littlewood-Paley theory (and its vector-valued version) and H\"older's 
\begin{eqnarray*}
& & \|\sum_k \La^{s_2-s_1}\De_k[V_{\sim k}\cdot \nabla  \vp_{<k+10}]\|_{L^p}\sim 
\|(\sum_k 2^{2k(s_2-s_1)} |\De_k[V_{\sim k}\cdot \nabla  \vp_{<k+10}]|^2)^{1/2}\|_{L^p} \leq \\
& \leq & C  \|(\sum_k 2^{2k(s_2-s_1)} |V_{\sim k}\cdot |\nabla  \vp_{<k+10}|^2)^{1/2}\|_{L^p}  \leq     \\
& \leq  & C
\|(\sum_k 2^{2k} |V_{\sim k}|^2)^{1/2}\|_{L^q}  \|\sup_k 2^{k(s_2-s_1-1)} |\nabla  \vp_{<k+10}|\|_{L^r}.
\end{eqnarray*}
Clearly, $\|(\sum_k 2^{2k} |V_{\sim k}|^2)^{1/2}\|_{L^q}\sim \|\nabla V\|_{L^q}$. For $s_2-s_1=1$, we have \\ 
$\|\sup_k 2^{k(s_2-s_1-1)} |\nabla  \vp_{<k+10}|\|_{L^r}\leq C \|\cm[\nabla \vp]\|_{L^r}\leq  C \|\La \vp\|_{L^r}$. 

For $s_2-s_1<1$, we can estimate  point-wise 
\begin{eqnarray*}
& & 2^{k(s_2-s_1-1)} |\nabla  \vp_{<k+10}| \leq   2^{k(s_2-s_1-1)} \sum_{l<k+10} |\nabla \vp_l |\leq \\
&\leq & C  \sum_{l<k+10} 2^{(l-k)(1-(s_2-s_1))} 2^{l(s_2-s_1)} 2^{-l} |\nabla \vp_l | \leq    C \cm[\La^{s_2-s_1} \vp].
\end{eqnarray*}
From these estimate, we conclude 
$$
\|\sup_k 2^{k(s_2-s_1-1)} |\nabla  \vp_{<k+10}|\|_{L^r}\leq C \| \cm[\La^{s_2-s_1} \vp]\|_{L^r}\leq C \|\La^{s_2-s_1}\vp\|_{L^r}.
$$

\subsubsection{High-high interactions} 
For this term, we need $s_1<1$ and $q>2$. 

Again that the commutator structure is not important and one term is simpler. So, we concentrate on 
$$
\La^{-s_1} \sum_k \De_k[\sum_{l=k+10}^\infty V_l \cdot \nabla \La^{s_2} \vp_{\sim l}]= \La^{-s_1} 
\sum_k \nabla  \De_k[\sum_{l=k+10}^\infty V_l \cdot \La^{s_2} \vp_{\sim l}]
$$
The contribution of these terms is bounded by 
$$
\|(\sum_k 2^{2k(1-s_1)} |\De_k[ \sum_{l=k+10}^\infty V_l \cdot \La^{s_2}\vp_{\sim l}]|^2)^{1/2} \|_{L^p}.
$$
By Littlewood-Paley theory the last expression is bounded by 
$$
I=\|(\sum_k 2^{2k(1-s_1)}  |\sum_{l=k+10}^\infty V_l \cdot \La^{s_2}\vp_{\sim l}|^2)^{1/2} \|_{L^p}.
$$
But  for $s_1<1$, we have 
\begin{eqnarray*}
 & & \sum_k 2^{2k(1-s_1)}  |\sum_{l=k+10}^\infty V_l \cdot \La^{s_2}\vp_{\sim l}|^2= \\
 &=&  
\sum_{l_1} V_{l_1} \cdot \La^{s_2}\vp_{\sim l_1} \sum_{l_2} V_{l_2} \cdot \La^{s_2}\vp_{\sim l_2} \sum_{k<\min(l_1, l_2)-10} 2^{2k(1-s_1)} \leq \\
&\leq & C \sum_{l_1} V_{l_1} \cdot \La^{s_2}\vp_{\sim l_1} \sum_{l_2} V_{l_2} \cdot \La^{s_2}\vp_{\sim l_2}  2^{2\min(l_1,l_2)(1-s_1)} \leq \\
&\leq & 
C \sum_{l} |V_{l}|^2 2^{2l}
 \sum_{l} |\La^{s_2-s_1}\vp_{\sim l} |^2.
\end{eqnarray*}
By H\"older's 
$$
I\leq \|(\sum_{l} |V_{l}|^2 2^{2l})^{1/2}\|_{L^q} \|(\sum_{l} |\La^{s_2-s_1}\vp_{\sim l} |^2)^{1/2}\|_{L^r}\leq C\|\nabla V\|_{L^q} \|\La^{s_2-s_1} \vp\|_{L^r}.
$$

In order to extend the results to  the case $a\in [s_2-s_1,1)$, it suffice to go over the different terms. For the low-high interaction term, we have, by our previous estimates 
\begin{eqnarray*}
 S^2(x)  &\leq &  C \sum_k 2^{2k(s_2-s_1)} |\cm[\cm[|\nabla V_{<k-10}|^{q_1}]^{1/q_1} 
 \cm[|\vp_{\sim k}|^{p_1}]^{1/p_1}]^2= \\
 &=& C \sum_k  |\cm[\cm[2^{k(s_2-s_1)} 2^{-k}|\nabla V_{<k-10}|^{q_1}]^{1/q_1} 
 \cm[2^k |\vp_{\sim k}|^{p_1}]^{1/p_1}]^2\leq \\
 &\leq & C \sum_k  |\cm[ \cm[\cm|\La^{s_2-s_1} V|^{q_1}]^{1/q_1}  \cm[2^k |\vp_{\sim k}|^{p_1}]^{1/p_1}]|^2
 \end{eqnarray*}
 Applying the Fefferman-Stein estimates yields (assuming $p_1<2, p_1<r, q_1<q$) 
 \begin{eqnarray*} 
 \|I_{low, high}\|_{L^p}\sim \|S\|_{L^p} &\leq &  C \| \cm[\cm|\La^{s_2-s_1} V|^{q_1}]^{1/q_1}\|_{L^q} \|( \sum_k  \cm[2^k |\vp_{\sim k}|^{p_1}]^{1/p_1})^{1/2}\|_{L^r}  \\
 &\leq & 
 C \|\La^{s_2-s_1} V\|_{L^q} \|\La \vp\|_{L^r}.
 \end{eqnarray*} 
An interpolation between the last estimate and \ref{36} yields the required estimate 
$$
\|I_{low, high}\|_{L^p}\leq \|\La^a V\|_{L^q} \|\La^{s_2-s_1+1-a}\vp\|_{L^r}. 
$$
Next, for the high-low terms, we  clearly have the following  bound 
$$
2^{k(s_2-s_1)} |\De_k[V_{\sim k}\cdot \nabla  \vp_{<k+10}]|(x)\leq C 
\cm[\cm[ \La^{s_2-s_1} V_{\sim_k}] \cm[\La^{1} \vp]], 
$$
Applying the same arguments as above  yields the bound \\ $\|I_{high, low}\|_{L^p} \leq C \|\La^{s_2-s_1} V\|_{L^q} \|\La^{1} \vp\|_{L^r}$, which by interpolation results in 
$$
\|I_{high, low}\|_{L^p} \leq C \|\La^a V\|_{L^q} \|\La^{s_2-s_1+1-a} \vp\|_{L^r}
$$
for all $a\in [s_2-s_1,1]$. 

Finally in the high-high case, one may move all the derivatives between $V$, $\vp$ (since they are both localized at the same frequency $l$), so in particular 
$$
\|I_{high, high}\|_{L^p}\leq C \|\La^a V\|_{L^q} \|\La^{s_2-s_1+1-a} \vp\|_{L^r}
$$

\subsection{Proof of \ref{25}} 
We start again with the low-high term. In this case, the estimate for $\|I_{low,high}\|_{L^2}$ is actually already contained in the estimates for $I_{low, high}$, since we have already remarked that in there, one can take $q=\infty$. 

Next, we verify the contribution of the high-low terms interactions.  We have by Littlewood-Paley theory   that 
\begin{eqnarray*}
& & \|I_{high, low}\|_{L^2}^2 \leq  C 
\sum_k 2^{2k(s_2-s_1-s_3)} \|V_{\sim k}\cdot \nabla  \vp_{<k+10}\|_{L^2}^2 \leq \\
&\leq & C \|V\|_{L^\infty}^2 
\sum_k 2^{2k(s_2-s_1-s_3)} \|\nabla  \vp_{<k+10}\|_{L^2}^2\leq \\
&\leq & C  \|V\|_{L^\infty}^2 
\sum_k 2^{2k(s_2-s_1-s_3)} \sum_{l<k+10} 2^{2l} \|\vp_{l}\|_{L^2}^2\\
&\leq & C \|V\|_{L^\infty}^2 \sum_l 2^{2l(1+s_2-s_1-s_3)} \|\vp_{l}\|_{L^2}^2 \leq C \|V\|_{L^\infty}^2 
\|\La^{s_2-s_1+1-s_3} \vp\|_{L^2}^2. 
\end{eqnarray*}
where in the derivation, we have used that 
$\sum_{k>l-10} 2^{2k(s_2-s_1-s_3)}\leq C 2^{2l(s_2-s_1-s_3)}$, which requires 
that $s_2-s_1-s_3<0$. 

Finally, we turn our attention to the high-high terms. Again, the commutator structure is unimportant here and we might as well consider the two terms separately. One of them is actually simpler (where $\La^{s_2}$ is acting on the low frequency outside), so we consider the other term only, namely 
$$
\La^{-s_1} \sum_k \De_k[\sum_{l=k+10}^\infty \La^{-s_3} V_l \cdot \nabla \La^{s_2} \vp_{\sim l}]= \La^{-s_1} 
\sum_k \nabla  \De_k[\sum_{l=k+10}^\infty \La^{-s_3}  V_l \cdot \La^{s_2} \vp_{\sim l}]
$$
Note that here again, we have moved $\nabla$ outside, because $\nabla\cdot V=0$. 
Taking $L^2$ norms yields 
\begin{eqnarray*}
& & \|I_{high, high}\|_{L^2}^2 \leq  C \sum_k 2^{2k(1-s_1)} \| \sum_{l=k+10}^\infty \La^{-s_3} V_l \cdot  \La^{s_2} \vp_{\sim l}\|_{L^2}^2 \leq \\
&\leq & C \|V\|_{L^{\infty}}^2 
\sum_k 2^{2k(1-s_1)}  (\sum_{l=k+10}^\infty 2^{l(s_2-s_3)} \|\vp_{\sim l}\|_{L^2})^2=\\
&=&  C \|V\|_{L^{\infty}}^2  \sum_{l_1}2^{l_1(s_2-s_3)} \|\vp_{\sim l_1}\|_{L^2} 
\sum_{l_2} 2^{l_2(s_2-s_3)} \|\vp_{\sim l_2}\|_{L^2} 
\sum_{k<\min(l_1, l_2)-10} 
2^{2k(1-s_1)}.
\end{eqnarray*}
Now, since $1-s_1>0$, we have 
$$
\sum_{k<\min(l_1, l_2)-10} 
2^{2k(1-s_1)}\leq C 2^{2\min(l_1, l_2)(1-s_1)} = C 2^{l_1(1-s_1)} 2^{l_2(1-s_1)} 2^{-|l_1-l_2|(1-s_1)}.
$$
Plugging this inside our estimate for $\|I_{high, high}\|_{L^2}^2 $ and applying Cauchy-Schwartz  we obtain 
\begin{eqnarray*}
\|I_{high, high}\|_{L^2}^2 & \leq &  C  \|V\|_{L^{\infty}}^2 \sum_{l_1, l_2} 2^{(l_1+l_2)(1-s_1+s_2-s_3)} 
\|\vp_{\sim l_1}\|_{L^2} 
 \|\vp_{\sim l_2}\|_{L^2} 2^{-|l_1-l_2|(1-s_1)}  \\
&\leq &  C \|V\|_{L^{\infty}}^2 \sum_{l_1, l_2} 2^{2l_1(1-s_1+s_2-s_3)} \|\vp_{\sim l_1}\|_{L^2}^2 
2^{-|l_1-l_2|(1-s_1)}\leq \\
&\leq& C\|V\|_{L^\infty}^2 \|\La^{1-s_1+s_2-s_3} \vp\|_{L^2}^2. 
\end{eqnarray*}
This concludes the proof of \ref{25} and thus of Lemma \ref{le:com1}.

          \end{document}